\documentclass[reqno,10pt]{amsart}
\usepackage{color}
\usepackage{amsmath}
\usepackage{amsfonts}
\usepackage{amssymb}
\usepackage{amsthm}
\usepackage{newlfont}
\usepackage{graphicx}

\newtheorem{thm}{Theorem}[section]
\newtheorem{cor}[thm]{Corollary}
\newtheorem{lem}[thm]{Lemma}
\newtheorem{prop}[thm]{Proposition}
\theoremstyle{definition}
\newtheorem{defn}[thm]{Definition}
\theoremstyle{remark}
\newtheorem{rem}[thm]{Remark}

\numberwithin{equation}{section}
\newcommand{\ed}{\end {document}}

\title[Hedgehog solutions]{Global wellposedness of hedgehog solutions
for the $(3+1)$ Skyrme model}
\author{Dong Li}
\email{madli@ust.hk}
\address{Department of Mathematics, The Hong Kong
University of Science \& Technology, Clear Water Bay, Hong Kong}


\begin{document}
\maketitle
\begin{abstract}
We consider the  hedgehog solutions in the
 $(3+1)$-dimensional Skyrme model which is an energy-supercritical problem.
We introduce a new strategy to prove global wellposedness
for arbitrarily large initial data.
\end{abstract}

\section{Introduction}
In this paper we consider the $(3+1)$-dimensional Skyrme model in quantum field theory.
This nonlinear sigma model was first proposed by Skyrme \cite{Sk61a, Sk61b, Sk62} to incorporate baryons
as stable field configurations in the description of low energy interaction of pions.
Let $U:\; \mathbb R^{3+1} \to SU(2)$ be a map into the isospin group with signature
$(+---)$. Define the $su(2)$-valued connection one-form $A$ by (below $U^{\dagger}$ denotes
the Hermitian adjoint)
\begin{align*}
A= U^\dagger  d U = A_\mu d x^\mu,
\end{align*}
where $x^0=t$, $(x^j)_{1\le j \le 3}= x \in \mathbb R^3$. The Lagrangian density of the classical
Skyrme model is given by
\begin{align}
\mathcal L= -\frac 1 4 f^2_{\pi} \text{Tr}(A_\mu A^\mu)
+ \frac 1 4 \epsilon^2 \text{Tr}\Bigl( [A_\mu, A_\nu] [A^\mu, A^\nu] \Bigr), \label{je1}
\end{align}
where $f_{\pi}^2$ is the pion decay constant, and $\epsilon>0$ is a  coupling parameter.
The actual value of $f_{\pi}^2$ does not play much role in our mathematical analysis and we will
conveniently set it to be $2$.
Here $[\cdot,\cdot]$ is the usual Lie bracket on $su(2)$ and $\text{Tr}(\cdot)$ denotes the matrix trace.

The Euler-Lagrangian equation of \eqref{je1} takes the form
\begin{align}
\partial_\mu \Bigl( A^\mu - \epsilon^2 [A_\nu, [A^\mu, A^\nu]] \Bigr)=0. \label{je2}
\end{align}

Let $I_2$ be the identity matrix and $\sigma_j$, $1\le j\le 3$ be the Pauli spin matrices.
Introducing the angular variable $\omega=\omega(t,x)$ and the spin vector $\mathbf n = (n_j) \in \mathbb S^2$,
we write the group element $U\in SU(2)$ as
\begin{align}
U(t,x) & = \exp \left( \frac{\omega(t,x)} {2i} \sigma_j n_j (t,x) \right) \notag \\
& = I_2 \cos \frac{\omega(t,x)} 2 - i \Bigl( \sigma_j n_j(t,x) \Bigr) \sin
\frac{\omega(t,x)} 2. \label{je3}
\end{align}

We shall be mainly concerned with a special family of solutions known as hedgehog solutions.
Under the hedgehog ansatz, we set $r=|x|$, $n_j(x) = \frac{x_j} r$ and $\omega(t,x)= 2f(r,t)$,
where $f$ is the unknown radial function. We then obtain from \eqref{je2}--\eqref{je3},
\begin{align}
  & \left(1+ \epsilon^2 \frac{2\sin^2 f} {r^2} \right) (\partial_{tt} -\partial_{rr}-\frac 2 r \partial_r )f \notag\\
  = & - \epsilon^2 \frac{4\sin^2 f}{r^3} \partial_r f - \epsilon^2 \frac{\sin(2f)} {r^2} \Bigl( (\partial_t f)^2 -(\partial_r f)^2
  \Bigr) \notag\\
  & \qquad - \frac{\sin(2f)} {r^2} - \epsilon^2 \frac{\sin^2 f \cdot \sin(2f)} {r^4}. \label{je4}
  \end{align}

Introduce the notations
\begin{align*}
\Delta_d = \partial_{rr} + \frac {d-1} r \partial_r
\end{align*}
and
\begin{align*}
\square_d = \partial_{tt} -\Delta_d = \partial_{tt} - \partial_{rr} - \frac{d-1} r \partial_r.
\end{align*}
For radial functions on $\mathbb R^d$, $\Delta_d$ and $\square_d$ are simply the usual Laplacian
and D'Alembertian in polar coordinates. In our work, it will be useful to lift the function $f(r)$
to a radial function in  $\mathbb R^d$ for some convenient choices of the dimension $d$.

Using the above notation, we write \eqref{je4} compactly as
\begin{align}
 \left(1+ \epsilon^2 \frac{2\sin^2 f} {r^2} \right) \square_3 f
  = & -\epsilon^2 \frac{4\sin^2 f}{r^3} \partial_r f -  \epsilon^2 \frac{\sin(2f)} {r^2} \Bigl( (\partial_t f)^2 -(\partial_r f)^2
  \Bigr) \notag\\
  & \qquad - \frac{\sin(2f)} {r^2} - \epsilon^2 \frac{\sin^2 f \cdot \sin(2f)} {r^4}. \label{je5}
\end{align}

The boundary conditions for $f$ are
\begin{align}
\lim_{r\to 0} f(t,r) = N_1\pi, \quad \lim_{r\to \infty} f(t,r)=0, \label{ge2}
\end{align}
where $N_1\ge 0$ is an integer.

The main result of this paper, roughly speaking, is that for smooth and arbitrarily large initial data the
corresponding solution to \eqref{je5}--\eqref{ge2} exists globally in time. The precise formulation of the
results will be given in Section \ref{sec_refor}. The basic conservation law associated with \eqref{je5}
is given by the Skyrme energy
\begin{align}
E(t) &= \frac 12 \int_{0}^\infty (1+ \epsilon^2 \frac{2\sin^2 f} {r^2} ) \Bigl( (\partial_t f)^2 +
(\partial_r f)^2 \Bigr) r^2 dr \notag \\
&\qquad+ \int_{0}^\infty \frac {\sin^2 f} {r^2}
(1+ \epsilon^2 \frac{\sin^2 f} {2r^2} ) r^2 dr \label{ge2a}\\
&=E_0, \qquad\forall\, t>0. \notag
\end{align}

 With respect to the Skyrme energy conservation, the main difficulty associated with the analysis of
\eqref{je5} is that it is \emph{energy-supercritical} and no useful
theory is readily available for such problems. We shall introduce a
new (and special) strategy to overcome this difficulty and prove
global wellposedness for arbitrarily large initial data. As far as
we know, this is the first unconditional result on a \emph{physical}
{energy-supercritical} problem.

We summarize below the main points of the proof.

\medskip
\noindent\textbf{Main steps of the proof}

\medskip
\noindent
In our analysis the value of $\epsilon$ does not play much role and we will henceforth set $\epsilon=1$
in \eqref{je5} for convenience.

\medskip
\noindent
Step $1$. Local (in time) analysis and lifting to dimension $5$.

The first step is to get a good local theory.
Observe that the nonlinearity on the RHS of \eqref{je5} has strong singularities near $r=0$ which can
only be balanced out by a good local asymptotics of $f$ as $r\to 0$. To kill this singularity we introduce
$g=g(r,t)$ by the relation
\begin{align}
 f(r,t) = \phi(r,t) + r g(r,t), \label{RVA1}
\end{align}
where $\phi$ is a smooth cut-off function such that $\phi(r)\equiv N_1\pi$ for $r\le 1$. We then
regard $g$ as a radial function on $\mathbb R^5$ and obtain from \eqref{je4}, \eqref{RVA1}
an equation for $g$ of the form
\begin{align*}
 \square_5 g =N(r,g,\partial_t g, \nabla g),
\end{align*}
where $N$ is a smooth nonlinearity and no longer contains any singularities near $r=0$. Local wellposedness
in  $H^k_{rad}(\mathbb R^5)$ then follows from energy estimates. From the local analysis, to continue the solution to all
time, we only need to control the quantity
\begin{align}
 G(t) = \left\| \langle x \rangle g(t,x) \right\|_{L_x^\infty(\mathbb R^5)} + \left\| \langle x \rangle
(|\partial_t g| + |\nabla g| ) \right \|_{L_x^\infty(\mathbb R^5)}. \label{RVA2}
\end{align}
We shall achieve this in several steps.

\medskip

\noindent
Step $2$. A nonlocal transformation and derivation of the $\Phi$-equation.

The blowup/continuation criteria \eqref{RVA2} is supercritical with respect to the Skyrme energy \eqref{ge2a}.
To nail down global wellposedness, we analyze in a deeper way the structure of \eqref{je5}. For this purpose, we
introduce a nonlocal transformation of the form (see Section 3 for more details)
\begin{align}
 \Phi(r,t) = \int_0^{g(r,t)} \left( 1 + \frac{2 \sin^2 (ry+\phi(r))}{r^2} \right)^{\frac12 } dy
+ \frac 1 {r^3} \phi_{\gtrsim 1}(r), \label{RVA3}
\end{align}
where $\phi_{\gtrsim 1}$ is a smooth cut-off function localized to the regime $r\gtrsim 1$.
Regard $\Phi$ as a radial function on $\mathbb R^5$.  For $\Phi$ we then obtain from \eqref{je5},
\eqref{RVA3} a nonlocal equation of the form
\begin{align}
 \square_5 \Phi= \frac 1 {r^3} \phi_{\gtrsim 1}-\frac 3 2 \Phi +
\frac 12 \int_0^{g(r,t)} \left( 3 B^{\frac 32} + B^{-\frac 12} -B^{-\frac 32} \right) dy, \label{RVA4}
\end{align}
where
\begin{align*}
 B= 1 + \frac{2 \sin^2 (ry + \phi(r))} {r^2}.
\end{align*}
The remarkable feature of this new system is that at the cost of nonlocality all derivative terms on the RHS of
\eqref{je4} have been eliminated.

\medskip

\noindent
Step $3$. Control of $H^1$-norm of $\Phi$ and a non-blowup argument.

This includes the estimates of $\|\Phi\|_{L_x^2(\mathbb R^5)}$, $\|\partial_t \Phi \|_{L_x^2 (\mathbb R^5)}$
and $\|\nabla \Phi \|_{L_x^2(\mathbb R^5)}$.
This is an important first step to beat energy supercriticality.
Due to the particular structure in \eqref{RVA3}, it is not difficult to check that the Skyrme energy \eqref{ge2a}
is insufficient to give any control of $\|\nabla \Phi \|_{L_x^2(\mathbb R^5)}$ which is a manifestation of energy supercriticality
at the lowest level.  A heuristic analysis (see the beginning of Section 4) shows that in the worst case scenario
the linear part of \eqref{RVA4} could take the form
\begin{align*}
 \square_5 \Phi =-\frac 3 2 \Phi + \frac 3 {r^2} \Phi
\end{align*}
which is a wave operator with \emph{negative} inverse square potential. Since $d=5$ and $3>\frac{(d-2)^2}4$, we
cannot use Strichartz (cf. \cite{BPST03}).  To solve this problem we resort to a nonlinear approach which exploits the fine structure of the equation.
Let $T$ be the first possible blowup time. By performing  estimates directly on \eqref{RVA3}--\eqref{RVA4}, we obtain
\begin{align}
 \int_{\mathbb R^5} \left( \frac 12 |\nabla \Phi(t)|^2
- \phi_{<r_0}(r) \cdot H(r,t) \right) dx \le C(T),
\quad \forall\, 0\le t<T, \label{RV_May11_1}
\end{align}
where $0<C(T)<\infty$ is a constant depending on $T$, $r_0<\frac 12$ is a small constant, $\phi_{<r_0}$ is a smooth
cut-off function
localized to $r\le r_0$, and
\begin{align*}
H(r,t) = \frac 32 \int_0^{g(r,t)} \left(1+\frac{2\sin^2(rw)}{r^2} \right)^{\frac 12} \cdot
\left(\int_0^w (1+\frac{2\sin^2(r y)}{r^2})^{\frac12} \cdot \frac{2 \sin^2(r y)}{r^2} dy \right) dw.
\end{align*}
By a detailed analysis on $H$, we  show that $H$ admits the sharp bound
\begin{align*}
H(r,t) \le \frac 94 \cdot \frac 12 \cdot \frac{|\Phi(r,t)|^2}{r^2}.
\end{align*}
From this and \eqref{RV_May11_1}, we get
\begin{align}
 0\le \int_{\mathbb R^5} \left( |\nabla \Phi(t)|^2 - \frac 94 \cdot \frac{|\Phi(t)|^2} {r^2} \right) dx  \le C(T),
\quad\forall\, 0\le t<T,
 \label{RVA5}
\end{align}
where the positivity of the integral follows from Hardy's inequality (see Lemma \ref{lem_Hardy})
 on $\mathbb R^5$. The estimate \eqref{RVA5} is the sharpest available and yet it
 is \emph{not coercive} enough to give control of $H^1$ norm
of $\Phi$. The main reason is that there could exist a sequence
\begin{align*}
 \| \nabla \Phi (t_n) \|_{L_x^2(\mathbb R^5)} \to +\infty, \quad
\left\| \frac{\Phi(t_n)} {r} \right\|_{L_x^2(\mathbb R^5)} \to +\infty,
\end{align*}
but
\begin{align*}
 \int_{\mathbb R^5} \left(  |\nabla \Phi(t_n)|^2 - \frac 94 \cdot \frac{|\Phi(t_n)|^2} {r^2} \right) dx \to C_1,
\quad \text{as $t_n \to T$},
\end{align*}
where $C_1\ge 0$ is a finite constant. To rule out this blowup scenario, we shall analyze in detail the special structure
of $\Phi$ and perform a delicate limiting and contradiction argument (see in particular \eqref{ge58}--\eqref{ge63} in the
proof of Proposition \ref{prop2}). The technical details are contained in the proof of Proposition \ref{prop2} and as a result
we can control the $H^1$-norm of $\Phi$.

\medskip

\noindent
Step $4$. Nonlinear energy bootstrap and higher order estimates.

In this final step we upgrade the $H^1$ estimate of $\Phi$ to $H^4$ estimates which are sufficient to give a priori
bound of the quantity $G(t)$ defined in \eqref{RVA2} (and yielding global wellposedness). The main task is to
interweave
the Sobolev estimates of $g$ and $\Phi$ back and forth a number of times using in an essential way the
structure of the nonlocal system
\eqref{RVA3}--\eqref{RVA4}. The estimates are organized in such a way that we first obtain temporal regularity
 and then use the structure of the equation to trade temporal regularity for spatial regularity.
The technical details are given in Section $5$.

\medskip
\noindent
The above four steps complete our proof of global wellposedness.
To put things into perspective, we briefly review below some results connected with the Skyrme model.

\medskip
\noindent
\textbf{Connection with other works}.

\begin{enumerate}

\item Prior to this work, progress has been slow on understanding the global dynamics of the
Skyrme model. In \cite{Wong11} Wong analyzed in detail the dominant
energy condition and the breakdown of hyberbolicity for the Skyrme
model (see also Gibbons \cite{Gibbons03},  Grutchfield and Bell
\cite{CB94}). In particular it follows that a small perturbation of
a static Skyrmion configuration yields local wellposedness. 
After our work is completed, the author learned that Geba, Nakanishi
and Rajeev \cite{GNR11} proved a small data global wellposedness and
scattering result for the Skyrme wave map for initial data in
critical Besov type space.

\item In \cite{GR10a, GR10b}, Geba and Rajeev considered a seimilinear Skyrme model introduced
by Adkins and Nappi \cite{AN84}. The equivariant solutions satisfy the following
\begin{align*}
\partial_{tt} f - \partial_{rr} f - \frac 2 r \partial_r f + \frac{\sin(2f)}{r^2}
+ \frac{(f-\sin f \cos f ) (1-\cos 2 f) } {r^4}=0
\end{align*}
and has conserved energy
\begin{align*}
E(f(t)) = \int_0^\infty \left( \frac 12 \Bigl( (\partial_t f)^2+(\partial_r f)^2 \Bigr)
+ \frac{\sin^2 f}{r^2} + \frac{(f-\sin f \cos f )^2} {2r^4} \right) r^2 dr.
\end{align*}
They proved that near the first possible blowup time, the energy does not
concentrate. But the issue of global wellposedness is still open.

\item If $\epsilon=0$ in \eqref{je5}, then we recover the equivariant wave map from $\mathbb R^{3+1}$ to
$\mathbb S^3$ which is also an \emph{energy-supercritical} problem. Generally smooth solutions will blow up in finite
time. Indeed Shatah \cite{Shatah88} constructed finite-time blowup solutions which is self-similar and
has finite energy. This was extended to other target manifolds in \cite{STZ94} and higher dimensions $d\ge 4$ in
\cite{CSTZ98}. In \cite{Bi00} Bizo\'n constructed a countable
family of spherically symmetric self-similar wave maps from the 3+1 Minkowski spacetime into the 3-sphere.
 These constructions all rely on the existence of a nontrivial harmonic map.

\item The $(2+1)$-dimensional analogue of the Skyrme model is known as baby Skyrme models.
The technique developed in this paper can also be used to prove global wellposededness of corresponding
hedgehog solutions. The details will be given in a future publication. In contrast, the $\epsilon=0$
limit of the baby Skyrme model gives rise to the $(2+1)$-dimensional \emph{energy-critical} equivariant wave map
\begin{align*}
\partial_{tt} f - \partial_{rr} f - \frac {\partial_r f}{r} + \frac{k^2 \sin(2f)} {2r^2}=0,
\end{align*}
where $k\ge 1$ is an integer giving the homotopy index. It is known that (cf. \cite{CTZ93,STZ94,Struwe03})
for smooth initial data with energy $E<E(Q)$, where $Q(r)=2\arctan(r^k)$, the corresponding solution
is global. Also by an argument of Struwe
 there is no blowup of self-similar type. The existence (and dynamics) of finite-time blowup
solutions were obtained in \cite{RS_Annals} ($k\ge 4$) and
\cite{KST_Inventione} ($k=1$)    
using different techniques and giving
different blowup rates. For results and some recent developments on
\emph{energy-critical} wave maps from $(2+1)$ Minkowski space to general target manifolds we refer to
\cite{KS97, KM97, Tataru01,
Bi01a, KM08,
ST10A, ST10B, KS09} and references therein.

\item The technique introduced in this paper has been recently generalised and extended
to many other important physical models. In \cite{Gnew1}, Geba and Grillakis have improved and
streamlined the result of this paper to Sobolev $H^s$, $s>7/2$. 
In \cite{Gnew2} and \cite{Gnew3}, Creek, and Geba-Grillakis have obtained large data global regularity for the $2+1$-dimensional equivariant Faddeev model.  We refer to the  monograph
\cite{Gnew4} for an extensive overview of more recent developments.

\end{enumerate}

\subsection*{Acknowledgements}
The author would like to thank Piotr Bizo\'n for some helpful
remarks and suggestions. The author would also like to thank
Dan-Andrei Geba, Kenji Nakanishi, Sarada G. Rajeev and Zhen Lei for
their interest in this work. 
The author is indebted to the anonymous referees for their many insightful remarks
and very helpful suggestions.
The author was supported in part by Hong Kong RGC grant GRF 16307317
and 16309518.

\section{Reformulation and main results} \label{sec_refor}
As was already mentioned, the value of $\epsilon$ will not play much role in our analysis as long as $\epsilon>0$.
In the rest of this paper we shall set $\epsilon=1$ in \eqref{je5}.

Denote
\begin{align*}
A_1 = 1+ \frac{2 \sin^2 f} {r^2}.
\end{align*}
Then
\begin{align}
\square_3 f & = -\frac 1 {A_1} \cdot \frac{4\sin^2 f} {r^3} \cdot \partial_r f
- \frac 1 {A_1} \cdot \frac{\sin(2f)}{r^2} \cdot \Bigl( (\partial_t f)^2 - (\partial_r f)^2 \Bigr) \notag \\
& \qquad -\frac 1 {A_1} \cdot \frac{\sin(2f)} {r^2} - \frac 1 {A_1} \cdot
\frac{\sin^2 f \cdot \sin(2f)} {r^4} \notag \\
& =: N(r, f, f^\prime), \label{ge1}
\end{align}

with boundary condition \eqref{ge2}.

Let $\phi$ be a smooth cut-off function such that $\phi(r) = N_1 \pi$ for $r\le 1$
and $\phi(r)=0$ for $r\ge 2$.

Define $g(r,t)$ by
\begin{align}
f(r,t) = \phi(r) + r g(r,t). \label{ge4}
\end{align}

No boundary condition is needed for $g$ at $r=0$.\footnote{We shall regard $g$ as
a radial function on $\mathbb R^5$ and construct a classical solution
$g \in H^k(\mathbb R^5)$. By radial Sobolev embedding, $|g(r,t)| \lesssim r^{-2}$
as $r\to \infty$. Hence the boundary condition $f(\infty, t)=0$ causes no
trouble either.}

Note that
\begin{align}
\square_3 f & = - \Delta_3  \phi + \square_3 (r g) \notag \\
&= -\Delta_3 \phi +r \square_5 g - \frac 2 r g. \label{ge5}
\end{align}

By \eqref{ge1}, \eqref{ge4}, \eqref{ge5}, the equation for $g$ then takes the form
\begin{align}
\square_5 g
&= \frac 2 {r^2} g + \frac 1 r \Delta_3 \phi
+ \frac 1 r \phi_{<1} \cdot N(r, rg, (rg)^\prime) \notag \\
& \qquad+ \frac 1 r \phi_{>1} \cdot N(r, \phi+rg, (\phi+rg)^\prime), \label{ge6}
\end{align}
where $\phi_{>1}=1-\phi_{<1}$, and
$\phi_{<1}$ is a smooth cut-off function such that $\phi_{<1}(r)=1$ for $r<\frac 12$; $\phi_{<1}(r)=0$ for
$r\ge 1$.

In more detail,
\begin{align}
\square_5 g & = \frac {\phi_{<1}} {1+\tilde F_0(rg) g^2} \Bigl( \tilde F_1(rg) g^3 + \tilde F_2(rg) g^5 \notag\\
&\qquad- \tilde F_3(rg) \cdot g \cdot \bigl( (\partial_t g)^2 - (\partial_r g)^2 \bigr) \notag \\
& \qquad + \tilde F_4(rg) \cdot g^4 \cdot r \partial_r g \Bigr) \notag \\
& \qquad + \phi_{>1} \cdot \frac 2 {r^2} g + \frac 1 r \Delta_3 \phi \notag \\
& \qquad + \frac 1 r \phi_{>1} \cdot N(r, \phi+rg, (\phi+rg)^\prime), \label{ge7}
\end{align}

where
\begin{align*}
\tilde F_0(x) & = 2 \Bigl( \frac{\sin x} x \Bigr)^2, \\
\tilde F_1(x) &= \frac 2 {x^2} - \frac{\sin(2x)} {x^3}, \\
\tilde F_2(x) & = \frac{\sin(2x)} {x^3} - \frac{\sin^2 x \sin (2x)}{x^5}, \\
\tilde F_3(x) & = \frac{\sin(2x)} {x}, \\
\tilde F_4(x) & = - \frac{4\sin^2 x} {x^4} + \frac{2 \sin(2x)} {x^3}.
\end{align*}

It is not difficult to check that $\tilde F_i(x)$, $0\le i \le 4$ are well-defined for all
$x\in \mathbb R$ with the help of power series expansion. Observe that the functions
$\tilde F_i$ can all be written as
\begin{align*}
\tilde F_i(x) = F_i(x^2), \quad i=0,\cdots, 4,
\end{align*}
where $F_i$ are smooth functions satisfying
\begin{align}
\left\| \frac{d^k}{dx^k} F_i(x) \right \|_{L_x^\infty} \le C_k, \quad \forall\, k\ge 0, \label{tmp_May_14_921}
\end{align}
where $C_k$ are constants depending only on $k$.

The reason that we write $\tilde F_i(rg) =F_i(r^2 g^2)$ is that we shall regard $F_i(r^2g^2) =F_i(|x|^2 g^2)$
for $x\in \mathbb R^5$ which is smooth in $x$.  This will help local energy estimates in the local theory.

Now  we lift $g$ to be radial function on $\mathbb R^5$, clearly then
\begin{align*}
r \partial_r g = \sum_{i=1}^5 x_i \cdot \partial_{x_i} g = x \cdot \nabla g.
\end{align*}

Thus we rewrite \eqref{ge7} as
\begin{align}
\square_5 g & = \frac{\phi_{<1}}{1+F_0(r^2 g^2) g^2} \Bigl( F_1(r^2 g^2) g^3 + F_2(r^2 g^2) g^5 \notag \\
& \qquad -F_3 (r^2 g^2) \cdot g \cdot \bigl( (\partial_t g)^2 -(\nabla g)^2 \bigr) \notag \\
& \qquad +F_4(r^2 g^2) \cdot g^4 \cdot ( x \cdot \nabla g) \Bigr) \notag \\
& \qquad + \phi_{>1} \cdot \frac 2 {r^2} g + \frac 1 r \Delta_3 \phi \notag \\
& \qquad + \frac 1 r \phi_{>1} \cdot N(r, \phi+rg, (\phi+rg)^\prime). \label{ge8}
\end{align}

\noindent
For any integer $k$, we shall denote by $H^k_{rad}(\mathbb R^5)$ the usual $H^k$ Sobolev space restricted to radial
functions on $\mathbb R^5$.

\begin{prop}[Local wellposedness and continuation criteria] \label{prop1}
Let $k>\frac 52+1$ be an integer. Assume
\begin{align*}
(g, \partial_t g) \Bigr|_{t=0} = (g_0,g_1) \in H^k_{\text{rad}} (\mathbb R^5) \times H^{k-1}_{\text{rad}}(\mathbb R^5).
\end{align*}
Then there exists $T>0$ and a local solution $g \in C([0,T), H_{\text{rad}}^k(\mathbb R^5)) \cap
C^1([0,T), H_{\text{rad}}^{k-1}(\mathbb R^5))$ to \eqref{ge8}. Furthermore the solution can be continued
past any $T_1\ge T$ as long as
\begin{align}
\sup_{0\le t <T_1} G(t) <\infty, \label{ge9}
\end{align}
where
\begin{align}
G(t) = \left\| \langle x \rangle g(t) \right\|_{L_x^\infty(\mathbb R^5)}
+ \left\| \langle x \rangle (|\partial_t g|+|\nabla g|) \right\|_{L_x^\infty(\mathbb R^5)}. \label{ge10}
\end{align}
\end{prop}

The proof of Proposition \ref{prop1} uses standard energy estimates and will be omitted here.
Our main result is

\begin{thm}[Global wellposedness for large data] \label{thm2}
Let $k\ge 4$ be an integer and assume
\begin{align*}
(g, \partial_t g ) \Bigr|_{t=0} = (g_0,g_1) \in H^k_{\text{rad}}  (\mathbb R^5) \times H^{k-1}_{\text{rad}}(\mathbb R^5).
\end{align*}
Then the corresponding solution in Proposition \ref{prop1} is global.
\end{thm}

By Proposition \ref{prop1}, the proof of Theorem \ref{thm2} reduces to showing that \eqref{ge9}
holds for any $T>0$. We shall achieve this by devising a new nonlinear energy bootstrap method.

\section{Nonlinear energy bootstrap: preliminary transformations}
Recall that \eqref{ge1} has the basic energy conservation
\begin{align}
E(t) & = \frac 1 2 \int_{0}^\infty  (1+ \frac{2\sin^2 f} {r^2} )
\Bigl( (\partial_t f)^2 + (\partial_r f)^2 \Bigr) r^2 dr \notag \\
& \qquad + \int_{0}^\infty \frac{\sin^2 f}{r^2}
\Bigl( 1+ \frac{\sin^2 f}{2r^2} \Bigr) r^2 dr \notag \\
& = E_0, \qquad \forall\, t>0. \label{ge11}
\end{align}

The continuation criteria \eqref{ge9} is supercritical with respect to this basic
energy conservation. To prove global wellposedness of \eqref{ge1} one certainly
needs a new strategy. In this section we explain the setup of our nonlinear energy
bootstrap argument.

Define $\tilde \Phi_1: (0,\infty) \times \mathbb R \to \mathbb R$ by
\begin{align}
\tilde \Phi_1 (\rho,z) = \int_{N_1 \pi}^z \Bigl( 1+ \frac{2\sin^2 y} {\rho^2} \Bigr)^{\frac 12 } dy.
\label{ge12}
\end{align}
The definition of $\tilde \Phi_1$ takes into consideration of the boundary condition \eqref{ge2} especially
when $N_1\ne 0$.

Define
\begin{align}
\Phi_1(r,t) = \tilde \Phi_1(r, f(r,t)). \label{ge13}
\end{align}

Then
\begin{align}
\square_3 \Phi_1 &= (\partial_{zz} \tilde \Phi_1)(r,f(r,t)) \Bigl( (\partial_t f)^2 - (\partial_r f)^2 \Bigr)
 + ( \partial_z \tilde \Phi_1)(r,f(r,t)) \square_3 f \notag\\
& \qquad - (\Delta_{3,\rho} \tilde \Phi_1)(r,f(r,t))
- 2 (\partial_{\rho} \partial_r \tilde \Phi_1) (r, f(r,t)) \partial_r f. \label{ge14}
\end{align}
Here $\Delta_{3,\rho}$ is the three-dimensional radial Laplacian in the $\rho$ variable, i.e.
\begin{align*}
(\Delta_{3,\rho} \tilde \Phi_1)(\rho, z)
= (\partial_\rho^2 \tilde \Phi_1)(\rho,z) + \frac 2 {\rho} (\partial_\rho \tilde \Phi_1)(\rho,z).
\end{align*}

Recall
\begin{align*}
A_1= 1+ \frac{2\sin^2 f}{r^2}.
\end{align*}

Easy to check that
\begin{align}
&(\partial_{zz} \tilde \Phi_1)(r,f(r,t))
+ (\partial_z \tilde \Phi_1)(r,f(r,t))
\cdot (- \frac 1 {A_1}) \cdot \frac{\sin(2f)} {r^2} =0, \notag \\
& -2(\partial_\rho \partial_z \tilde \Phi_1)(r,f(r,t))
-\frac {(\partial_z \tilde \Phi_1)(r,f(r,t))} {A_1} \cdot \frac{4\sin^2 f}{r^3}=0. \label{ge14a}
\end{align}

Therefore by \eqref{ge14}, \eqref{ge1} and \eqref{ge14a}, we get
\begin{align}
\square_3 \Phi_1 & = -A_1^{-\frac 12}
\cdot \frac{\sin(2f)} {r^2} - A_1^{-\frac 12}
\cdot \frac{\sin^2 f \cdot \sin(2f)} {r^4}
 - (\Delta_{3,\rho} \tilde \Phi_1) (r, f(r,t)). \label{ge15}
\end{align}

Denote
\begin{align}
B_1 = 1+ \frac{2\sin^2 y} {r^2}. \label{ge15a}
\end{align}

By a simple computation,
\begin{align*}
\Delta_{3,r} ( B_1^{\frac 12} ) = \frac 1 {r^2} \Bigl( B_1^{-\frac 12} - B_1^{-\frac 32} \Bigr).
\end{align*}

Hence
\begin{align}
(\Delta_{3,\rho} \tilde \Phi_1)(r,f(r,t))
= \frac 1 {r^2} \int_{N_1 \pi}^{f(r,t)}
\Bigl( B_1^{-\frac 12} - B_1^{-\frac 32} \Bigr) dy. \label{ge16}
\end{align}

By a tedious calculation, we have
\begin{align}
A_1^{-\frac 12} \cdot \frac{\sin(2f)} {r^2}
& = \frac 1 {r^2} \int_{N_1 \pi}^{f(r,t)} \partial_y
\Bigl( B_1^{-\frac 12} \cdot \sin(2y) \Bigr) dy \notag \\
& = \frac 1 {r^2}
\int_{N_1\pi}^{f(r,t)}
B_1^{-\frac 32}
\Bigl( 2-r^2 (B_1^2-1) \Bigr) dy. \label{ge17}
\end{align}

Similarly
\begin{align}
A_1^{-\frac 12} \cdot
\frac{\sin^2 f \cdot \sin(2f)}{r^4}
& = \frac 1 {r^2} \cdot \int_{N_1\pi}^{f(r,t)}
\Bigl( 2B_1^{\frac 12} - B_1^{-\frac 32} -B_1^{-\frac 12} \Bigr) dy \notag \\
& \qquad + \frac 12 \int_{N_1\pi}^{f(r,t)}
B_1^{-\frac 32} (-3B_1^3 + 5B_1^2 -B_1-1) dy. \label{ge18}
\end{align}

Plugging \eqref{ge16},  \eqref{ge17} and \eqref{ge18} into \eqref{ge15}, we obtain
\begin{align}
\square_3 \Phi_1
 = -\frac 2 {r^2} \Phi_1
+ \frac 12 \int_{N_1\pi}^{f(r,t)}
\Bigl( 3B_1^{\frac 32} -3B_1^{\frac 12} + B_1^{-\frac 12} -B_1^{-\frac 32} \Bigr) dy.
\label{ge20}
\end{align}

Equation \eqref{ge20} is still not very satisfactory since it contains terms of inverse
square potential type. To remove such terms, one more transformation is needed.

Define $\Phi_2(r,t)$ by
\begin{align}
\Phi_1(r,t) = r \Phi_2(r,t). \label{ge20a}
\end{align}

Then
\begin{align}
\square_3 \Phi_1 &= \square_3 (r \Phi_2) \notag \\
& = r \square_5 \Phi_2 - \frac 2 {r^2} \Phi_1. \label{ge21}
\end{align}

By \eqref{ge21}, equation \eqref{ge20} expressed in the $\Phi_2$ variable now takes
the form
\begin{align}
\square_5 \Phi_2
= -\frac 32 \Phi_2
+ \frac 1{2r} \int_{N_1\pi}^{f(r,t)}
\Bigl( 3 B_1^{\frac 32} + B_1^{-\frac 12} -B_1^{-\frac 32} \Bigr) dy. \label{ge22}
\end{align}

Although formally the RHS of \eqref{ge22} still contains $1/r$ terms which may be singular
when $r\to 0$, it actually causes no trouble in our energy bootstrap estimates later.
To see this, we bring back the $g$-function used in the local analysis.

Recall that
\begin{align}
f(r,t) = \phi(r) + r g(r,t), \label{ge23a}
\end{align}
where $\phi(r)\equiv N_1\pi$ for $r<1$ and $\phi(r)=0$ for $r\ge 2$.

Define
\begin{align}
B_2 = 1 + \frac{2\sin^2 (ry)} {r^2}. \label{ge24a}
\end{align}
Observe that $B_2$ is a smooth function (see the discussion preceding the estimate \eqref{tmp_May_14_921}).

Let $\phi_{<1}$ be a smooth cut-off function such that $\phi_{<1}(r)=1$ for $r\le \frac 12$ and
$\phi_{<1}(r)=0$ for $r>1$. By \eqref{ge23a}, \eqref{ge24a}, we have
\begin{align}
 &\frac 1{2r} \phi_{<1}(r) \int_{N_1\pi}^{f(r,t)}
\Bigl( 3B_1^{\frac 32} + B_1^{-\frac 12} -B_1^{-\frac 32} \Bigr) dy \notag \\
 =&  \frac 1{2r} \phi_{<1}(r) \int_{N_1 \pi}^{N_1 \pi+r g(r,t)}
\Bigl( 3B_1^{\frac 32} + B_1^{-\frac 12} -B_1^{-\frac 32} \Bigr) dy \notag \\
=&  \frac 12 \phi_{<1}(r) \int_{0}^{g(r,t)}
\Bigl( 3B_2^{\frac 32} + B_2^{-\frac 12} -B_2^{-\frac 32} \Bigr) dy. \label{ge25a}
\end{align}

In the second equality above, we have performed a change of variable
$y \to N_1\pi +ry$. Clearly \eqref{ge25a} is smooth as long as $g$ is smooth since
it has no singular terms in $r$.

By using \eqref{ge25a}, we rewrite \eqref{ge22} as
\begin{align}
\square_5 \Phi_2 & = -\frac 32 \Phi_2
 + \frac 12 \phi_{<1} \int_0^{g(r,t)}
\Bigl( 3 B_2^{\frac 32} +B_2^{-\frac 12} - B_2^{-\frac 32} \Bigr) dy \notag \\
& \qquad + \frac 1 {2r}
\phi_{>1} \int_{N_1\pi}^{f(r,t)}
\Bigl( 3B_1^{\frac 32} +B_1^{-\frac 12} -B_1^{-\frac 32} \Bigr) dy, \label{ge23}
\end{align}
where $\phi_{>1}=1-\phi_{<1}$ is localized to $r\gtrsim 1$.

Equation \eqref{ge23} is almost good for us since it no longer contains any
derivative terms or singularities in $r$. However there is one more problem.

By \eqref{ge12}, \eqref{ge13}, \eqref{ge20a}, we have
\begin{align}
\Phi_2(r,t) = \frac 1r \int_{N_1\pi}^{f(r,t)}
\Bigl( 1+ \frac{2 \sin^2 y} {r^2} \Bigr)^{\frac 12} dy. \label{ge24}
\end{align}

By \eqref{ge23} it is not difficult to check that $\Phi_2$ has no singularity
near $r\sim 0$. However for $r\ge 2$ by using energy conservation \eqref{ge11}
and radial Sobolev embedding, we get
$|f(r,t)| \lesssim r^{-1}$. If $N_1>0$, then \eqref{ge24} asserts that
\begin{align*}
\Phi_2(r,t) \sim \frac{Const}{r}, \quad\text{as $r\to \infty$}.
\end{align*}
In particular $\Phi_2 \notin L_x^2(\mathbb R^5)$ when we regard $\Phi_2$ as a radial
function on $\mathbb R^5$. We therefore need to introduce one more transformation to kill this
divergence.

To this end, we define
\begin{align}
\Phi(r,t) &= \Phi_2(r,t) +
\frac 1 3 \phi_{>1} \cdot \frac 1 {r}
\int_0^{N_1\pi} \Bigl( 3 B_1^{\frac 32} +B_1^{-\frac 12} -B_1^{-\frac 32} \Bigr)dy
\label{ge25}\\
 & = \frac 1 r \phi_{<1} \int_{N_1 \pi}^{f(r,t)} B_1^{\frac 12} dy  \label{ge26} \\
& \qquad + \frac 1r \phi_{>1} \int_{0}^{f(r,t)} B_1^{\frac 12} dy
\label{ge27} \\
& \qquad + \frac 1 r \phi_{>1}
\int_0^{N_1\pi}
\Bigl( B_1^{\frac 32} - B_1^{\frac 12}
+\frac 13 B_1^{-\frac 12} - \frac 13 B_1^{-\frac 32} \Bigr) dy. \label{ge28}
\end{align}

Since $\phi(r)\equiv N_1\pi$ for $r<1$, by \eqref{ge23a}, we have
\begin{align}
\eqref{ge26}
 = \phi_{<1} \int_0^{g(r,t)}
\Bigl( 1+ \frac{2\sin^2(ry+\phi(r))}{r^2} \Bigr)^{\frac 12} dy. \label{ge29}
\end{align}

For \eqref{ge27}, we have
\begin{align}
\eqref{ge27}
 = \phi_{>1} \int_0^{g(r,t)} \Bigl( 1+ \frac{2\sin^2(ry+\phi(r))}{r^2} \Bigr)^{\frac 12} dy
 + \frac 1 r \phi_{>1} \int_0^{\phi(r)} B_1^{\frac 12} dy. \label{ge30}
\end{align}

Note that $\phi(r)=0$ for $r\ge 2$, therefore we can write
\begin{align}
\frac 1 r \phi_{>1} \int_0^{\phi(r)} B_1^{\frac 12} dy = \phi_{\sim 1} (r), \label{ge31}
\end{align}
where $\phi_{\sim 1}$ is a smooth cut-off function localized to $r\sim 1$.

For \eqref{ge28}, observe that by \eqref{ge15a}
\begin{align*}
B_1^{\frac 32} - B_1^{\frac 12} = O\left( \frac 1 {r^2} \right), \quad r \gtrsim 1,
\end{align*}
and similarly
\begin{align*}
\frac 13 B_1^{-\frac 12} - \frac 1 3 B_1^{-\frac 32} = O \left( \frac 1 {r^2} \right),
\quad r \gtrsim 1.
\end{align*}

Therefore we shall write
\begin{align}
\eqref{ge28} = \frac 1 {r^3} \phi_{\gtrsim 1}(r), \label{ge32}
\end{align}
where $\phi_{\gtrsim 1}(r)$ is a smooth cut-off function localized to $r\gtrsim 1$ and can
vary from place to place.

By using \eqref{ge26}--\eqref{ge32}, we obtain
\begin{align*}
\Phi(r,t) = \int_0^{g(r,t)} \left(1+ \frac{2 \sin^2(ry+\phi(r))} {r^2} \right)^{\frac 12} dy
 + \phi_{\sim 1} + \frac 1 {r^3} \phi_{\gtrsim 1}(r).
\end{align*}

We can further include $\phi_{\sim 1}(r)$ into $\phi_{\gtrsim 1}(r)$ and simply write
\begin{align*}
\phi_{\sim 1}(r) + \frac 1 {r^3} \phi_{\gtrsim 1}(r) = \frac 1 {r^3} \phi_{\gtrsim 1}(r).
\end{align*}

Then
\begin{align}
\Phi(r,t)  = \int_0^{g(r,t)} \left(1+ \frac{2\sin^2(ry+\phi(r))} {r^2} \right)^{\frac 12} dy
 + \frac 1 {r^3} \phi_{\gtrsim 1}(r). \label{ge33}
\end{align}

On the other hand by \eqref{ge25} and a simple computation, we have
\begin{align}
\Phi(r,t) = \Phi_2(r,t) + \frac 1 r \cdot \phi_{\gtrsim 1}(r). \label{ge34}
\end{align}

Plugging \eqref{ge34} into \eqref{ge23} and using \eqref{ge25}, we get
\begin{align}
\square_5 \Phi & = \frac 1 {r^3} \cdot \phi_{\gtrsim 1} - \frac 32 \Phi
 + \frac 12 \phi_{<1} \int_0^{g(r,t)} \Bigl( 3B^{\frac 32}_2 +B_2^{-\frac 12} - B_2^{-\frac 32} \Bigr) dy
\notag \\
& \qquad + \frac 12 \cdot \phi_{>1} \cdot \frac 1 r
\int_0^{f(r,t)} \Bigl( 3 B_1^{\frac 32} + B_1^{-\frac 12} -B_1^{-\frac 32} \Bigr) dy. \label{ge35}
\end{align}

By using an argument similar to the derivation of \eqref{ge33}, we further simplify \eqref{ge35} as
\begin{align}
\square_5 \Phi  = \frac 1 {r^3} \phi_{\gtrsim 1} - \frac 32 \Phi
+ \frac 12 \int_0^{g(r,t)}
\Bigl( 3 B^{\frac 32} + B^{-\frac 12} - B^{-\frac 32} \Bigr) dy, \label{ge36}
\end{align}
where
\begin{align}
B= 1 + \frac{2\sin^2(ry +\phi(r))} {r^2}. \label{ge37}
\end{align}

Formula \eqref{ge33} then takes the form
\begin{align}
\Phi(r,t) = \int_0^{g(r,t)} B^{\frac 12} dy + \frac 1 {r^3} \phi_{\gtrsim 1}(r). \label{ge38}
\end{align}

We analyze \eqref{ge36}--\eqref{ge38} in the next section.

\section{Non-blowup of $H^1$-norm of $\Phi$}

The first step in our analysis is to control the $H^1$-norm of
$\Phi$. This includes $\|\Phi\|_{L_x^2(\mathbb R^5)}$, $\|\partial_t \Phi \|_{L_x^2(\mathbb R^5)}$
and $\| \nabla \Phi \|_{L_x^2(\mathbb R^5)}$. By \eqref{ge38}, \eqref{ge4}, we have
\begin{align*}
\partial_t \Phi = \frac 1 r \cdot \partial_t f \cdot (1+ \frac{2\sin^2 f} {r^2} )^{\frac 12},
\end{align*}
and therefore by \eqref{ge11}, we get
\begin{align}
\| \partial_t \Phi \|_{L_x^2(\mathbb R^5)} \lesssim 1. \label{May14_eq24}
\end{align}
By \eqref{ge37}, \eqref{ge38}, easy to see
\begin{align}
|\Phi(r,t)| \lesssim |g(r,t)| + |g(r,t)|^2 + \frac 1 {r^3} |\phi_{\gtrsim 1}(r)|. \label{May14_eq25}
\end{align}
By the assumption of Proposition \ref{prop1} and Sobolev embedding, we have $\| g(0)\|_{L_x^4(\mathbb R^5)} \lesssim 1$.
By \eqref{May14_eq25}, this gives $\|\Phi(0)\|_{L_x^2(\mathbb R^5)} \lesssim 1$. Using \eqref{May14_eq24}, we then have
\begin{align}
\| \Phi(t) \|_{L_x^2(\mathbb R^5)} \le Const\cdot t, \quad\forall\, t>0. \label{May14_eq26}
\end{align}

By \eqref{ge11}, we have
\begin{align*}
\| \partial_t f \|_{L_x^2(\mathbb R^3)} + \| \partial_r f \|_{L_x^2(\mathbb R^3)} \lesssim 1.
\end{align*}
Since $\| f(0)\|_{L_x^2(\mathbb R^3)} \lesssim 1$, we get
\begin{align}
\| f(t) \|_{H_x^1(\mathbb R^3)} \le Const \cdot t, \quad\forall\, t>0.\label{ge39}
\end{align}
By \eqref{ge4} and Hardy's inequality (see \eqref{ge47}), we get
\begin{align}
\| g(t) \|_{H_x^1(\mathbb R^5)} \le Const \cdot t, \quad\forall\, t>0. \label{ge40}
\end{align}

However it is not difficult to check that
 \eqref{ge11} and \eqref{ge40} are insufficient to bound $\| \nabla \Phi\|_{L_x^2(\mathbb R^5)}$.

One may try to do Strichartz. But there is one problem as we now explain.

Imagine that
\begin{align}
g(r,t) \sim \frac 1 r, \label{ge41}
\end{align}
\emph{for a range of values of $r\ll 1$}.\footnote{Certainly \eqref{ge41} cannot hold
for all $r\to 0$ since $g$ is assumed to be regular at $r=0$.}

Then by \eqref{ge38},
\begin{align*}
\Phi(r,t) \sim \frac{\sqrt 2} r g (r,t),
\end{align*}
and
\begin{align*}
\frac 3 2 \int_0^{g(r,t)} B^{\frac 32} dy & \sim \frac{3\sqrt 2} {r^3} g(r,t) \\
& \sim \frac 3 {r^2} \Phi(r,t).
\end{align*}

Therefore for a range of values of $r\ll1$, the linear part of \eqref{ge36} takes
the form
\begin{align}
\square_5 \Phi = - \frac 3 2 \Phi + \frac 3 {r^2} \Phi. \label{ge42}
\end{align}

Equation \eqref{ge42} is a wave operator with \emph{negative} inverse square potential.
Since $d=5$ and
\begin{align*}
3> \frac{(d-2)^2} 4,
\end{align*}
no Strichartz is available (cf. \cite{BPST03}). This destroys the hope
of employing good linear estimates.

Therefore a new idea is required to establish $H^1$-norm bound of $\Phi$. In particular we
shall use a nonlinear approach which exploits in an essential way the structure of the equation.

\begin{lem} \label{lem1}
There exists $r_0>0$ sufficiently small, such that for any $0\le r \le r_0$, we have
\begin{align*}
F(\beta) = \int_0^\beta (r^2+2\sin^2 y)^{\frac 12}
\Bigl( \frac 3 4 - \sin^2 y \Bigr) dy \ge 0, \quad \forall\, \beta\ge 0.
\end{align*}
If $r>0$, then the equality holds iff $\beta=0$.
\end{lem}

\begin{proof}[Proof of Lemma \ref{lem1}]
By a simple calculation, we have
\begin{align*}
\int_0^{\pi} (\sin y) \cdot \Bigl( \frac 3 4 - \sin^2 y \Bigr) dy = \frac 16.
\end{align*}
Clearly there exists $r_1>0$ sufficiently small such that
\begin{align}
\int_0^{\pi}
(r^2+2\sin^2 y)^{\frac 12}
\Bigl( \frac 34 -\sin^2 y \Bigr)dy \ge \frac 1 {12}, \quad \forall\, 0\le r <r_1. \label{fe1}
\end{align}

Consider $m\pi \le \beta<(m+1)\pi$ and $m$ is large. Then by \eqref{fe1}, for $0\le r <r_1$, we have
\begin{align*}
& \int_0^{\beta} (r^2+ 2\sin^2 y)^{\frac 12} \Bigl( \frac 34- \sin^2 y \Bigr) dy \\
\ge & \frac 1 {12} m + \int_{m\pi}^\beta (r^2+2\sin^2 y)^{\frac 12} \Bigl( \frac 34 -\sin^2 y\Bigr) dy \\
\ge & \frac{m}{12} -O(1) > \frac 1{12},
\end{align*}
if $m$ is taken to be sufficiently large.

Therefore we only need to consider $F(\beta)$ on a compact interval $[0,m\pi]$.

Observe that $F(0)=0$, $F(m\pi)>\frac 1{12}$. It suffices to consider critical points of $F$
in $(0,m\pi)$ and prove the positivity of $F$ at these points.

Solving $F^\prime(\beta)=0$ yields
\begin{align*}
\sin(\beta) = \pm \frac{\sqrt 3} 2.
\end{align*}
Hence
\begin{align*}
\beta= j\pi+ \frac {\pi}3 \quad \text{or}\quad j\pi+\frac{2\pi}3, \quad j\ge 0, \, j\in \mathbb Z.
\end{align*}

If $\beta=j\pi+ \frac{\pi}3$, then for $0\le r <r_1$, by \eqref{fe1},
\begin{align*}
F(j\pi+ \frac{\pi} 3) & = \int_0^{j\pi+\frac{\pi}3} (r^2+ 2\sin^2 y)^{\frac 12}\Bigl( \frac 34 -\sin^2 y \Bigr) dy \notag \\
& \ge \frac{j}{12} + \int_0^{\frac {\pi} 3} (r^2+2\sin^2 y)^{\frac 12} \Bigl( \frac 34-\sin^2 y \Bigr) dy \\
& >0.
\end{align*}

If $\beta= j\pi+ \frac{2\pi} 3$, then for $0\le r <r_1$,
\begin{align*}
F(j\pi+ \frac{2\pi}3 ) & \ge \frac j{12} + \int_0^{\frac{2\pi}3}
(r^2+ 2\sin^2 y)^{\frac 12} \Bigl( \frac 34 -\sin^2 y \Bigr) dy \\
& \ge \int_0^{\frac{2\pi}3} (r^2+2\sin^2 y)^{\frac 12} \Bigl( \frac 34 -\sin^2 y \Bigr) dy.
\end{align*}

Define
\begin{align*}
\tilde F(\rho) = \int_0^{\frac{2\pi}3} (\rho + \sin^2 y)^{\frac 12} \Bigl(\frac 34 -\sin^2 y \Bigr) dy.
\end{align*}

Easy to check $\tilde F(0)=0$.

On the other hand,
\begin{align*}
\frac{\tilde F(\rho) -\tilde F(0)} {\rho} & = \int_0^{\frac{2\pi}3}
\frac 1 {\sqrt{\rho+\sin^2 y} + \sqrt{\sin^2 y} } \cdot \Bigl( \frac 34 -\sin^2 y \Bigr) dy \\
&>0, \qquad \text{for $\rho$ sufficiently small.}
\end{align*}
Hence  $F(j\pi+\frac{2\pi}3)>0$ for $0<r\le r_0$, where $r_0$ is sufficiently small.
\end{proof}

Define $G_i: \; (0,\infty) \times \mathbb R \to \mathbb R$, $i=0,1,2$ by
\begin{align}
G_0(r, w) & = \int_0^w \Bigl( 1+ \frac{2\sin^2 (r y)} {r^2} \Bigr)^{\frac 12}
\cdot \frac{2\sin^2(r y)} {r^2} dy, \label{ge43} \\
G_1(r,z) & = \frac 32 \int_0^z G_0(r,w) \Bigl( 1+ \frac{2\sin^2(r w)}{r^2} \Bigr)^{\frac 12} dw,
\label{ge44} \\
G_2(r,w) & = \int_0^w \Bigl( 1+ \frac{2\sin^2(r y)}{r^2} \Bigr)^{\frac 12} dy. \label{ge45}
\end{align}

\begin{cor} \label{cor1}
For any $0<r\le r_0$, $z\in \mathbb R$, we have
\begin{align}
\Bigl| G_1(r,z) \Bigr| \le \frac 9 4 \cdot \frac 1 2 \cdot
\frac{\bigl( G_2(r,z) \bigr)^2}{r^2}. \label{ge46}
\end{align}
\end{cor}

\begin{proof}[Proof of Corollary \ref{cor1}]
Since $G_1(r,z)$ is an even function of $z$, it suffices to consider the case
$z>0$. By Lemma \ref{lem1} for $w\ge 0$,
\begin{align*}
0\le \frac 32 G_0(r,w) \le \frac 9 4 \cdot \frac 1 {r^2} \cdot G_2(r,w).
\end{align*}
Therefore
\begin{align*}
0\le G_1(r,z) & \le \frac 94 \cdot \frac 1 {r^2} \cdot \int_0^z G_2(r,w)
\cdot \Bigl( 1+ \frac{2\sin^2(r w)}{r^2} \Bigr)^{\frac 12 } dw \\
& = \frac 94 \cdot \frac 1 {r^2} \int_0^z G_2(r,w)
\cdot (\partial_w G_2)(r,w) dw \\
& = \frac 9  4 \cdot \frac 12 \cdot \frac{\bigl( G_2(r,z) \bigr)^2}{r^2}.
\end{align*}
\end{proof}

\begin{lem}[Hardy's inequality] \label{lem_Hardy}
Let $d\ge 3$. Then
\begin{align}
\int_{\mathbb R^d} \frac{f^2}{|x|^2} dx \le \frac 4 {(d-2)^2}
\int_{\mathbb R^d} |\nabla f|^2 dx, \quad \forall\, f\in C_0^\infty(\mathbb R^d). \label{ge47}
\end{align}
The constant $\frac 4 {(d-2)^2}$ is sharp.
\end{lem}

The goal of this section is to prove the following

\begin{prop}[Non-blowup of $H^1$ norm of $\Phi$] \label{prop2}
Let $T>0$ be the maximal lifespan of the local solution $g$ constructed in
Proposition \ref{prop1}. If $T<\infty$, then
\begin{align}
\sup_{0\le t <T}
\Bigl( \| \Phi(t) \|_{H_x^1(\mathbb R^5)} + \| \partial_t \Phi(t) \|_{L_x^2(\mathbb R^5)}
\Bigr) <\infty. \label{ge48}
\end{align}
\end{prop}

Before we begin the proof of Proposition \ref{prop2}, we set up some notations.

\begin{align*}
\end{align*}
\noindent
\textbf{Notation}. Throughout the rest of this paper, unless explicitly mentioned,
we shall suppress the dependence of constants on the initial data or on the time $T$.
For example we shall write \eqref{ge48} simply as
\begin{align*}
\| \Phi(t) \|_{H_x^1(\mathbb R^5)} + \| \partial_t \Phi(t) \|_{L_x^2(\mathbb R^5)}
\lesssim 1, \quad \forall\, 0\le t <T.
\end{align*}

\begin{proof}[Proof of Proposition \ref{prop2}]

By \eqref{May14_eq24}, \eqref{May14_eq26}, we only need to show
\begin{align*}
\| \nabla \Phi (t)\|_{L_x^2(\mathbb R^5)} \lesssim 1, \quad\forall\, 0\le t<T.
\end{align*}

Let $\psi \in C_c^\infty(\mathbb R^5)$, $0\le \psi \le 1$ be a radial smooth
cut-off function such that $\psi(x) =1$ for $|x| \le \frac 12$ and $\psi(x) =0 $
for $|x|\ge 1$. Choose $r_0 \le \frac 12$ as in Lemma \ref{lem1} and define
\begin{align*}
\phi_{<r_0}(x) & = \psi\bigl( \frac  x {r_0} \bigr), \\
\phi_{>r_0}(x) & = 1- \phi_{<r_0}(x).
\end{align*}

By \eqref{ge36}--\eqref{ge38}, we have
\begin{align}
\square_5 \Phi &= \frac 1 {r^3} \phi_{\gtrsim 1} +
\frac 12 \int_0^{g(r,t)} \Bigl( B^{-\frac 12} - B^{-\frac 32} \Bigr) dy
+ \frac 32 \int_0^{g(r,t)} \Bigl( B^{\frac 32} -B^{\frac 12} \Bigr) dy \notag \\
& = \frac 1 {r^3} \phi_{\gtrsim 1}
+ \frac 1 2 \int_0^{g(r,t)} \Bigl( B^{-\frac 12} -B^{-\frac 32} \Bigr) dy
 + \frac 32 \phi_{>r_0} \int_0^{g(r,t)} \Bigl( B^{\frac 32} -B^{\frac 12} \Bigr) dy \notag \\
& \qquad +
\frac 32 \phi_{<r_0} \int_0^{g(r,t)}
\Bigl( 1+ \frac{2\sin^2(ry)} {r^2} \Bigr)^{\frac 12}
\cdot \frac{2\sin^2(ry)} {r^2} dy. \label{ge49}
\end{align}

Multiplying both sides of \eqref{ge49} by $\partial_t \Phi$ and integrating by parts, we obtain
\begin{align}
& \frac d {dt} \int_{\mathbb R^5}
\Bigl( \frac 12 (\partial_t \Phi)^2
+ \frac 1 2 |\nabla \Phi|^2 - \phi_{<r_0}(x)
\cdot G_1(r,g(r,t)) \Bigr) dx \notag \\
\lesssim & \| \partial_t \Phi \|_{L_x^2(\mathbb R^5)} \cdot
\Bigl(1+ \| g(t) \|_{L_x^2(\mathbb R^5)} \Bigr). \label{ge50}
\end{align}

Plugging \eqref{May14_eq24}, \eqref{ge40} into \eqref{ge50} and integrating in time, we get
\begin{align}
& \sup_{0\le t <T}
\int_{\mathbb R^5}
\Bigl( \frac 12 (\partial_t \Phi)^2 + \frac 12 |\nabla \Phi|^2
 - \phi_{<r_0}(x) \cdot G_1(r,g(r,t)) \Bigr) dx \lesssim 1. \label{ge53}
\end{align}

In particular, this yields
\begin{align}
\int_{\mathbb R^5}
\Bigl( \frac 1 2 |\nabla \Phi(t) |^2
-\phi_{<r_0}(x) \cdot G_1(r,g(r,t)) \Bigr) dx \lesssim 1,
\quad \forall\, 0\le t <T. \label{ge54}
\end{align}

Using Corollary \ref{cor1} and \eqref{May14_eq26}, we get
\begin{align}
\int_{\mathbb R^5} \Bigl( |\nabla \Phi(t)|^2
-\frac 9 4 \cdot \frac{|\Phi(t)|^2}{r^2} \Bigr) dx \lesssim 1,
\quad \forall\, 0\le t <T. \label{ge55}
\end{align}

By Hardy's inequality (Lemma \ref{lem_Hardy}), we have
\begin{align}
\int_{\mathbb R^5} \Bigl(
|\nabla \Phi(t)|^2 - \frac 9 4 \cdot \frac{|\Phi(t)|^2} {r^2}
\Bigr) dx \ge 0. \label{ge56}
\end{align}

There is no hope to obtain \eqref{ge48} by using only \eqref{ge55},
\eqref{ge56}, since there could possibly exist a sequence
$\Phi(t_n)$ with the property that
\begin{align*}
\| \nabla \Phi(t_n) \|_{L_x^2(\mathbb R^5)} \to \infty, \quad
\left\| \frac{\Phi(t_n)} {r} \right\|_{L_x^2(\mathbb R^5)} \to \infty,
\end{align*}
but
\begin{align*}
\int_{\mathbb R^5} \Bigl( |\nabla \Phi(t_n) |^2 - \frac 9 4 \cdot
\frac{|\Phi(t_n)|^2} {r^2} \Bigr) dx \to C_1, \quad \text{as $t_n\to
T$},
\end{align*}
where $C_1\ge 0$ is a finite constant.

Certainly a new argument is needed here.

To solve this problem, we shall proceed by exploiting in more detail the structure of $\Phi$.

Assume \eqref{ge48} does not hold. By \eqref{May14_eq24} and \eqref{May14_eq26}, there exists $t_n \to T$ such that
\begin{align}
\lim_{n\to \infty} \| \nabla \Phi(t_n) \|_{L_x^2(\mathbb R^5)} =+\infty. \label{ge58}
\end{align}

Define
\begin{align}
\tilde \Phi(t_n) = \frac{\Phi(t_n)} { \left\| \nabla \Phi(t_n) \right\|_{L_x^2(\mathbb R^5)} }. \label{ge59a}
\end{align}

Then
\begin{align}
\left\| \nabla \tilde \Phi(t_n) \right\|_{L_x^2(\mathbb R^5)} =1, \label{ge59}
\end{align}
and by \eqref{ge58}, \eqref{ge55}, \eqref{ge56},
\begin{align}
\left\| \frac{\tilde \Phi(t_n) } {r} \right\|_{L_x^2(\mathbb R^5)} \to \frac 23,
\quad \text{as $t_n \to T$}. \label{ge60}
\end{align}

Next consider \eqref{ge38}. If $r\gtrsim 1$, then
\begin{align*}
| \partial_r \Phi(r,t) | \lesssim |\partial_r g| + \left| \frac g r \right| +
\left| \frac 1 {r^3} \phi_{\gtrsim 1}(r) \right|.
\end{align*}
Therefore by \eqref{ge40} and \eqref{May14_eq26},
\begin{align}
\left\| \nabla \Phi(t) \right\|_{L_x^2(|x|>\frac 12,\; x\in \mathbb R^5)}
+\left\| \frac 2 r \Phi(t) \right\|_{L_x^2(|x|>\frac 12,\; x\in \mathbb R^5)}
 \lesssim 1.
\label{ge61}
\end{align}

For $r\le \frac 12$, by \eqref{ge38} and a short computation, we have
\begin{align}
&\partial_r \Phi(r,t) + \frac 2 r \Phi (r,t) \notag \\
 =&  \left( 1+ \frac{2\sin^2 f} {r^2} \right)^{\frac 12}
\cdot \frac{\partial_r f} r
 + \frac 1 r \int_0^{g(r,t)} \left( 1+ \frac{2\sin^2(ry)} {r^2} \right)^{-\frac 12} dy. \label{ge62}
\end{align}

By \eqref{ge11},
\begin{align*}
\left\| \left(1+ \frac{2\sin^2 f} {r^2} \right)^{\frac 12} \cdot \frac{\partial_r f} r
\right\|_{L_x^2(\mathbb R^5)} \lesssim 1.
\end{align*}

Hence \eqref{ge61}, \eqref{ge62} gives
\begin{align}
\left\| \partial_r \Phi(t) + \frac 2 r \Phi(t)
\right\|_{L_x^2(\mathbb R^5)} \lesssim 1, \quad \forall\, 0\le t <T. \label{ge63}
\end{align}

By \eqref{ge58}, \eqref{ge59a}, \eqref{ge63}, we obtain
\begin{align*}
\left\|
\partial_r \tilde \Phi (t_n)
+ \frac 2 r \tilde \Phi(t_n) \right\|_{L_x^2(\mathbb R^5)} \to 0,
\quad \text{as $t_n\to T$}.
\end{align*}
But this contradicts  \eqref{ge59} and \eqref{ge60}.
\end{proof}

\begin{rem}
In the above derivation, the contradiction (blow-up) argument is actually not
needed. One can directly use \eqref{ge55} and \eqref{ge63} to get the desired uniform
bound on $\| \partial_r \Phi(t) \|_{L_x^2(\mathbb R^5)}$. We thank the anonymous
referee for pointing this out.
\end{rem}

\section{Nonlinear energy bootstrap: more estimates}

Let $T>0$ be the same as in Proposition \ref{prop2}. Our goal in this
section is to prove
\begin{align}
\sum_{|\alpha|+|\beta|\le 4}
\left\| \partial_x^\alpha \partial_t ^\beta \Phi(t) \right\|_{L_x^2(\mathbb R^5)} \lesssim 1,
\quad \forall\, 0\le t <T, \label{he1}
\end{align}
and eventually
\begin{align}
\sup_{0\le t <T} G(t) <\infty, \label{he2}
\end{align}
where $G(t)$ is defined in \eqref{ge10}. By Proposition \ref{prop1}, this implies
global wellposedness.

We shall prove \eqref{he1} in several steps.

First we get some decay estimates of $\Phi$ and $g$.

By Proposition \ref{prop2} and radial Sobolev embedding, we have
\begin{align}
|\Phi(r,t)| \lesssim \min \left\{
r^{-\frac 32}, \, r^{-2} \right\},
\quad\forall\, r>0,\; 0\le t<T. \label{he3}
\end{align}

We claim that
\begin{align}
|g(r,t)| \lesssim
\min \left\{ r^{-\frac 34}, \; r^{-2} \right\},
\quad\forall\, r>0, \; 0\le t<T. \label{he4}
\end{align}

By \eqref{ge33}, it suffices to prove
\begin{align*}
|g(r,t)| \lesssim r^{-\frac 34}, \quad \forall\, 0<r\ll 1, \; 0\le t<T.
\end{align*}

For $r\ll 1$, \eqref{ge38} gives
\begin{align}
\Phi(r,t) = \int_0^{g(r,t)}
\left( 1+ \frac{2\sin^2(ry)} {r^2} \right)^{\frac 12} dy. \label{he5}
\end{align}

Suppose for some $0<r\ll 1$, $|g(r,t)| \gtrsim \frac 1r$, then clearly
\begin{align*}
\Phi(r,t) \sim \frac g r.
\end{align*}
By \eqref{he3}, this would imply
\begin{align*}
|g(r,t)| \lesssim r^{-\frac 12},
\end{align*}
which contradicts  the assumption $|g(r,t)| \gtrsim \frac 1r$.

Therefore $|g(r,t)| \lesssim \frac 1r$ for all $r\ll 1$. By \eqref{he5}, we get
\begin{align*}
|\Phi(r,t)| \sim \left|\int_0^{g(r,t)} (1+y^2)^{\frac 12} dy \right| \gtrsim g(r,t)^2.
\end{align*}

Hence by \eqref{he3},
\begin{align*}
g(r,t)^2 \lesssim r^{-\frac 32}, \quad \forall\, 0<r\ll 1, \; 0\le t<T.
\end{align*}

Therefore \eqref{he4} is proved.

Before we continue, we need to introduce standard Strichartz for the wave operator.

\begin{defn}
Let $d\ge 2$. A pair $(q,r)$ is said to be wave admissible  if
\begin{align*}
2\le q\le \infty, \quad 2 \le r <\infty, \\
\text{and} \qquad \frac 1 q + \frac{d-1}{2r} \le \frac{d-1} 4.
\end{align*}
Note that the case $(q,r,d)=(2,\infty,3)$ is not admissible.
\end{defn}

\begin{lem} \label{lem_S1}
Let  $d\ge 2$.  Suppose $u: \; [0,T] \times \mathbb R^d \to \mathbb R$ solves
\begin{align*}
\begin{cases}
\partial_{tt} u -\Delta u = F, \\
(u,\partial_t u ) \Bigr|_{t=0} = (u_0, u_1).
\end{cases}
\end{align*}
Let $(q,r)$, $(\tilde q, \tilde r)$ be wave admissible and satisfy the scaling condition
\begin{align*}
\frac 1 q + \frac d r = \frac d 2 - \gamma = \frac 1 {\tilde q^\prime}
+ \frac d {\tilde r^\prime} -2.
\end{align*}
Then on the space-time slab $[0,T]\times \mathbb R^d$, we have
\begin{align*}
& \| u \|_{L_t^q L_x^r} + \| u \|_{C_t \dot H_x^\gamma} +
\| \partial_t u \|_{C_t \dot H_x^{\gamma-1}} \\
\lesssim & \|u_0 \|_{\dot H_x^\gamma}
+ \| u_1\|_{\dot H_x^{\gamma-1}}
+ \| F \|_{L_t^{\tilde q^\prime} L_x^{\tilde r^\prime}}.
\end{align*}
Here $(\tilde q^\prime, \tilde r^\prime)$ are the conjugates of $(\tilde q,\tilde r)$, i.e.
$\frac 1 {\tilde q}+ \frac 1 {\tilde q^\prime} = \frac 1 {\tilde r}
+ \frac 1 {\tilde r^\prime} =1$.
\end{lem}

To simplify the presentation, we introduce more

\noindent
\textbf{Notations}. For any $z\in \mathbb R^d$, we use the Japanese bracket notation $\langle z\rangle :=(1+|z|^2)^{\frac 12}$.
For any space-time slab $[0,T_1]\times \mathbb R^5$, we shall use
the notation
\begin{align*}
\| u \|_{L_t^q L_x^r([0,T_1])}
\end{align*}
to denote
\begin{align*}
\| u \|_{L_t^q L_x^r([0,T_1]\times \mathbb R^5)}.
\end{align*}

We will need to use the standard Littlewood-Paley projection operators.
Let $\tilde \phi \in C_c^\infty(\mathbb R^5)$ be a radial bump function supported in the ball
$\{ x\in \mathbb R^5:\; |x| \le \frac{25}{24} \}$ and equal to one on the ball
$\{ x\in \mathbb R^5: \; |x| \le 1\}$. For any constant $C>0$, denote
$\tilde \phi_{\le C}(x) := \tilde \phi(\frac x C)$ and $\tilde \phi_{>C} := 1- \tilde \phi_{\le C}$.
For each dyadic $N>0$, define the Littlewood-Paley projectors
\begin{align*}
\widehat{P_{\le N} f}(\xi) & := \tilde \phi_{\le N}(\xi) \hat f(\xi), \\
\widehat{P_{>N} f }(\xi) & := \tilde \phi_{>N}(\xi) \hat f(\xi), \\
\widehat{P_N f} (\xi) & := (\tilde \phi_{\le N} - \tilde \phi_{\le \frac N2} ) \hat f(\xi),
\end{align*}
and similarly $P_{<N}$ and $P_{\ge N}$.

Now we are ready to continue our estimates.

Taking the time derivative on both sides of \eqref{ge36}, we get
\begin{align}
\square_5 (\partial_t \Phi) & = -\frac 32 \partial_t \Phi+ \frac 32 A^{\frac 32} \partial_t g
+ \frac 12 (A^{-\frac 12} - A^{-\frac 32} ) \partial_t g , \label{he6}
\end{align}
where
\begin{align}
A= 1 + \frac{2 \sin^2( rg(r,t)+\phi(r) )} {r^2}. \label{he7}
\end{align}

By \eqref{he4}, we have
\begin{align}
|A-1| \lesssim \min\left\{ r^{-\frac 32}, \; r^{-4} \right\}. \label{he8}
\end{align}

From \eqref{ge38}, one has
\begin{align}
\partial_t \Phi = A^{\frac 12} \partial_t g. \label{he9}
\end{align}

Substituting \eqref{he9} into \eqref{he6}, we get
\begin{align}
\square_5 ( \partial_t \Phi)
=\Bigl( \frac 32 + \frac 12 A^{-2} \Bigr) (A-1) \partial_t \Phi. \label{he9a}
\end{align}

By Strichartz (Lemma \ref{lem_S1}) and \eqref{he8}, we have for any
$0<T_1<T$,
\begin{align}
\left\| P_{\ge 1} \partial_t \Phi
\right\|_{L_t^3 L_x^3 ([0,T_1])} &
\lesssim \left \| P_{\ge 1} \partial_t \Phi(0) \right\|_{\dot H_x^{\frac 12} }
 + \left\| P_{\ge 1} \partial_{tt} \Phi(0)
\right \|_{\dot H_x^{-\frac 12} } \notag \\
& \qquad + \left\| (A-1) \partial_t \Phi \right \|_{L_t^{\frac 32} L_x^{\frac 32} ([0,T_1])} \notag \\
& \lesssim 1 + \| (A-1) \|_{L_t^3L_x^3([0,T_1])} \cdot
\| \partial_t \Phi \|_{L_t^3 L_x^3([0,T_1])} \notag \\
& \lesssim 1+ T_1^{\frac 13} \| \partial_t \Phi \|_{L_t^3 L_x^3 ([0,T_1])}. \label{he10}
\end{align}

Obviously
\begin{align}
\| P_{<1} \partial_t \Phi
\|_{L_t^3 L_x^3([0,T_1])} \lesssim \| \partial_t \Phi\|_{L_t^{\infty}L_x^2} \lesssim 1. \label{he11}
\end{align}

Using \eqref{he10}, \eqref{he11}, a continuity argument (see Appendix
\ref{S:cont}) yields
\begin{align}
\left \| \partial_t \Phi \right \|_{L_t^3 L_x^3 ([0,T)) } \lesssim 1. \label{he12}
\end{align}

Therefore
\begin{align}
\left \| \partial_t \Phi \right \|_{L_t^\infty \dot H_x^{\frac 12} ([0,T))} \lesssim 1. \label{he13}
\end{align}

Using \eqref{he9a}, we have
\begin{align}
\square_5 (\partial_{tt} \Phi )
& = \left( \frac 32 + \frac 12 A^{-2} \right) (A-1) \partial_{tt} \Phi \notag \\
& \qquad + ( -\frac 12 A^{-2} + A^{-3} + \frac 32 ) \partial_t A \partial_t \Phi. \label{he13a}
\end{align}

By \eqref{he7}, observe that
\begin{align}
|\partial_t A| \lesssim |\partial_t \Phi|. \label{he13aa}
\end{align}

Therefore by \eqref{he12},
\begin{align*}
 & \left \| ( -\frac 12 A^{-2} +A^{-3} + \frac 32 ) \partial_t A \partial_t \Phi \right\|_{L_t^{\frac 32} L_x^{\frac 32}([0,T))} \\
 \lesssim & \left \| \partial_t \Phi \right\|_{L_t^3 L_x^3([0,T))}^2 \lesssim 1.
 \end{align*}

Denote
\begin{align}
G_3(r,t) = \frac 12 \int_0^{g(r,t)}
\Bigl( 3B^{\frac 32} + B^{-\frac 12} -B^{-\frac 32} \Bigr) dy. \label{he16a}
\end{align}

Then by \eqref{ge37}, \eqref{he4}, we have
\begin{align}
|G_3(r,t)| \lesssim
\begin{cases}
|g(r,t)|, \qquad \text{if $r\gtrsim 1$}, \\
|\Phi(r,t)|^2 + |\Phi(r,t)|, \qquad \text{if $r\ll 1$}.
\end{cases} \label{he16}
\end{align}

Hence by \eqref{ge48},
\begin{align}
\left\| P_{<1} G_3(t) \right\|_{L_x^2(\mathbb R^5)} \lesssim 1, \quad \forall\, 0\le t <T. \label{he16b}
\end{align}

 By essentially repeating the derivation of \eqref{he12}, \eqref{he13} with $\partial_t \Phi$ replaced by
 $\partial_{tt} \Phi$, we get
 \begin{align}
 \left \| \partial_{tt} \Phi \right \|_{L_t^\infty  H_x^{\frac 12} ([0,T))} \lesssim 1. \label{he14}
 \end{align}
 Note that the low frequency part of $\partial_{tt} \Phi$ causes no trouble since it
 can be controlled by $\| P_{\le 1} \Delta \Phi\|_{L_x^2} \lesssim \| \Phi\|_{L_x^2}$ using
 equation \eqref{ge36} together with \eqref{he16b}.

Now by \eqref{ge36}, we have
\begin{align}
-\Delta \Phi = -\partial_{tt} \Phi+ \frac 1 {r^3} \phi_{\gtrsim 1} -\frac 32 \Phi + G_3, \label{he15}
\end{align}
where $G_3(r,t)$ was already defined in \eqref{he16a}.

By \eqref{he16} and \eqref{ge40},
\begin{align}
\left\| G_3(t) \right \|_{L_x^2(|x|>\frac 14,\; x\in \mathbb R^5)} \lesssim 1, \quad \forall\,
0\le t <T. \label{he17}
\end{align}

By \eqref{he16} and \eqref{ge48}, we get
\begin{align}
\left\| \phi_{\le \frac 12} G_3(t)
\right\|_{L_x^{\frac 53}(\mathbb R^5)}
\lesssim 1,
\quad \forall\, 0\le t <T, \label{he18}
\end{align}
where $\phi_{\le \frac 12}$ is a smooth cut-off function localized to $r\le \frac 12$.

By \eqref{he14}, \eqref{he15}, \eqref{he17} and \eqref{he18}, we have
\begin{align*}
\left\| P_{>1} |\nabla|^{-\frac 12} \Delta \Phi (t)
\right\|_{L_x^2 (\mathbb R^5)} & \lesssim \left \| P_{>1} \partial_{tt} \Phi(t) \right \|_{\dot H_x^{\frac 12} (\mathbb R^5)}
 + 1 + \left\| |\nabla|^{-\frac 12} P_{>1} G_3 \right \|_{L_x^2(\mathbb R^5)} \\
&  \lesssim 1 + \left \| (1-\phi_{\le \frac 12} ) G_3(t) \right \|_{L_x^2 (\mathbb R^5)}
 + \left \| \phi_{\le \frac 12} G_3(t) \right \|_{L_x^{\frac 53} (\mathbb R^5) } \\
& \lesssim 1, \qquad \forall \, 0\le t<T.
\end{align*}

Hence
\begin{align*}
\left \| |\nabla|^{\frac 32} \Phi(t) \right\|_{L_x^2(\mathbb R^5)} \lesssim 1, \qquad \forall\, 0\le t<T.
\end{align*}
By Sobolev embedding,
\begin{align}
\| \Phi (t)\|_{L_x^5 (\mathbb R^5)} \lesssim 1, \qquad\forall\, 0\le t<T. \label{he19}
\end{align}

By \eqref{he16}, \eqref{he17}, \eqref{he19}, we get
\begin{align*}
\left\| G_3(t) \right\|_{L_x^2 (\mathbb R^5)} \lesssim 1, \quad\forall\, 0\le t<T.
\end{align*}

Hence by \eqref{he15} and \eqref{he14}, we obtain
\begin{align}
\left \| \Phi(t) \right \|_{H_x^2 (\mathbb R^5)} \lesssim 1, \qquad\forall\, 0\le t <T. \label{he20}
\end{align}

By radial Sobolev embedding, we have
\begin{align*}
\left\| r^{\frac 12} \Phi(t) \right\|_{L_x^\infty} \lesssim
\left\| \Delta \Phi \right \|_{L_x^2} \lesssim 1.
\end{align*}

Therefore \eqref{he3}, \eqref{he4} and \eqref{he8} can be refined to
\begin{align}
|\Phi(r,t)| & \lesssim \min \left\{ r^{-\frac 12}, \, r^{-2} \right\}, \label{he21} \\
|g(r,t) |& \lesssim \min \left\{ r^{-\frac 14}, \, r^{-2} \right\}, \label{he22} \\
|A-1| & \lesssim \min\left\{ r^{-\frac 12}, \, r^{-4} \right\}. \label{he23}
\end{align}

By \eqref{he13a}, \eqref{he13aa} and Strichartz, we have for any $T_1<T$,
\begin{align}
& \left\| \partial_{tt} \Phi \right\|_{L_t^\infty \dot H_x^1([0,T_1))}
+ \left\| \partial_{ttt} \Phi \right \|_{L_t^\infty L_x^2 ([0,T_1])}  \notag \\
\lesssim & \left\| \partial_{tt} \Phi(0) \right\|_{\dot H_x^1} +
\left\| \partial_{ttt} \Phi(0) \right \|_{L_x^2}
 + \left\| (A-1) \partial_{tt} \Phi \right \|_{L_t^1 L_x^2([0,T_1])} \notag \\
& \qquad + \left \| \partial_t A \cdot \partial_t \Phi \right\|_{L_t^1 L_x^2([0,T_1])} \notag \\
\lesssim & 1 + T_1 \left \| (A-1 ) \right\|_{L_t^\infty L_x^5([0,T_1])}
\cdot \left\| \partial_{tt} \Phi \right \|_{L_t^\infty \dot H_x^1([0,T_1])} \notag \\
& \qquad + \left \| \partial_t \Phi \right \|^2_{L_t^2 L_x^4 ([0,T_1])}. \label{he24}
\end{align}

By \eqref{he23},
\begin{align}
\left\| (A-1) \right \|_{L_t^\infty L_x^5} \lesssim 1. \label{he25}
\end{align}

By \eqref{he9a}, \eqref{he23} and Strichartz, it is not difficult to check that
\begin{align}
\left \| \partial_t \Phi \right\|_{L_t^2 L_x^4([0,T))}
+\left\| \partial_t \Phi \right\|_{L_t^\infty \dot H_x^1([0,T))}
\lesssim 1. \label{he26}
\end{align}

Plugging \eqref{he25}, \eqref{he26} into \eqref{he24}, a simple continuity argument
then shows that
\begin{align}
\left\| \partial_{tt} \Phi
\right \|_{L_t^\infty \dot H_x^1([0,T))}
+ \left\| \partial_{ttt} \Phi
\right\|_{L_t^\infty L_x^2([0,T))} \lesssim 1.
\label{he27}
\end{align}

By \eqref{he9a}, \eqref{he27}, \eqref{he23} and Hardy's inequality, we then have
\begin{align}
\left\| \partial_t \Delta \Phi
\right\|_{L_t^\infty L_x^2([0,T))}
& \lesssim \left \| \partial_{ttt} \Phi \right \|_{L_t^\infty L_x^2([0,T))}
 + \left\| (A-1) \partial_t \Phi \right \|_{L_t^\infty L_x^2([0,T))} \notag \\
& \lesssim 1+ \left \| \nabla \partial_t \Phi \right \|_{L_t^\infty L_x^2 ([0,T))} \notag \\
& \lesssim 1. \label{he28}
\end{align}

We can write \eqref{he27}, \eqref{he28} collectively as
\begin{align}
 \left\| \partial_{ttt} \Phi \right\|_{L_t^\infty L_x^2([0,T))}
+ \left \| \partial_{tt} \nabla \Phi \right\|_{L_t^\infty L_x^2([0,T))}
+ \left\| \partial_t \Delta \Phi
\right \|_{L_t^\infty L_x^2([0,T))} \lesssim 1. \label{he29}
\end{align}

By \eqref{ge36} and \eqref{he29}, we have
\begin{align}
 & \left\| \nabla \Delta \Phi \right\|_{L_t^\infty L_x^2([0,T))} \notag \\
 \lesssim & 1+ \left\| \partial_r
 \left( \int_0^{g(r,t)}
 \Bigl( 3B^{\frac 32} + B^{-\frac 12} -B^{-\frac 32} \Bigr) dy \right) \right\|_{L_t^\infty L_x^2 ([0,T))} \notag \\
 \lesssim & 1+ \left\| A^{\frac 32} \partial_r g \right \|_{L_t^\infty L_x^2([0,T))} \notag \\
 & \qquad + \left\|
 \int_0^{g(r,t)}
 \Bigl( \frac 9 2 B^{\frac 12} - \frac 12 B^{-\frac 32} + \frac 32 B^{-\frac 52} \Bigr)
 \partial_r B dy \right\|_{L_t^\infty L_x^2 ([0,T))}. \label{he30}
 \end{align}

 Observe that for $r\le \frac 12$,
 \begin{align*}
 |(\partial_r B)(r,y)| \lesssim |y|^3.
 \end{align*}
 Therefore by \eqref{he22},
 \begin{align}
 & \left \| \int_0^{g(r,t)}
 \Bigl( \frac 9 2 B^{\frac 12} - \frac 12 B^{-\frac 32} + \frac 32 B^{-\frac 52} \Bigr)
 \partial_r B dy \right\|_{L_t^\infty L_x^2 ([0,T))} \notag \\
 \lesssim & \;\| g\|^5_{L_t^\infty L_x^{10}([0,T))}
 + \| g \|_{L_t^\infty L_x^2([0,T))} \notag \\
 \lesssim &\; 1. \label{he31}
 \end{align}

 On the other hand, by \eqref{he22}, \eqref{he7} and \eqref{ge40},
 \begin{align}
  \left\| A^{\frac 32} \partial_r g \right \|_{L_t^\infty L_x^2([0,T))}
 \lesssim & \left \| \partial_r g \right \|_{L_t^\infty L_x^2([0,T))}
 + \left\| \phi_{<\frac 12} r^{-\frac 34} \partial_r g \right \|_{L_t^\infty L_x^2([0,T))} \notag \\
 \lesssim & 1 + \left\| \phi_{<\frac 12}
 \cdot r^{-\frac 34} \partial_r g \right \|_{L_t^\infty L_x^2([0,T))}. \label{he32}
 \end{align}

 Plugging \eqref{he31}, \eqref{he32} into \eqref{he30}, we get
 \begin{align}
 & \left\| \nabla \Delta \Phi \right \|_{L_t^\infty L_x^2([0,T))}
 \lesssim  1 +
 \left\| \phi_{< \frac 12} \cdot r^{-\frac 34} \cdot \partial_r g
 \right\|_{L_t^\infty L_x^2([0,T))}. \label{he33}
 \end{align}

 By \eqref{ge62}, \eqref{he21}, \eqref{he22} and Hardy's inequality
 (see Appendix \ref{S:add1}), we have
 \begin{align*}
 \left\| \phi_{< \frac 12} \cdot r^{-\frac 34} \cdot \partial_r g
 \right\|_{L_t^\infty L_x^2([0,T))}
 & \lesssim 1+ \left\| \phi_{<\frac 12} \cdot r^{-\frac 34} \partial_r \Phi \right\|_{L_t^\infty L_x^2} \\
 & \lesssim 1 + \left\| \frac 1 r \nabla \Phi \right\|_{L_t^\infty L_x^2} \\
 & \lesssim 1.
 \end{align*}

 Substituting it into \eqref{he33}, we get
 \begin{align*}
 \left\| \nabla \Delta \Phi \right \|_{L_t^\infty L_x^2([0,T))} \lesssim 1.
 \end{align*}

 Hence together with \eqref{he29}, we have
 \begin{align}
 &\left\| \partial_{ttt} \Phi \right\|_{L_t^\infty L_x^2 ([0,T))} + \left\| \partial_{tt} \nabla \Phi \right
 \|_{L_t^\infty L_x^2 ([0,T))} \notag \\
 & \qquad + \left\|
 \partial_t \Delta \Phi \right\|_{L_t^\infty L_x^2([0,T))}
 + \left\| \nabla \Delta \Phi \right \|_{L_t^\infty L_x^2([0,T))} \lesssim 1. \label{he34}
 \end{align}

By Sobolev embedding we get
\begin{align*}
\left\| \Phi \right \|_{L_t^\infty L_x^\infty([0,T))} \lesssim 1.
\end{align*}

Therefore we refine \eqref{he21}, \eqref{he22}, \eqref{he23} to
\begin{align}
|\Phi(r,t) | & \lesssim \langle r \rangle^{-2}, \label{he35} \\
|g(r,t) | & \lesssim \langle r \rangle^{-2}, \label{he36} \\
|A-1| & \lesssim \langle r\rangle^{-4}. \label{he37}
\end{align}

By \eqref{ge62} (see Appendix \ref{S:add1}), we get
\begin{align}
\left\| \partial_r g \right\|_{L_t^\infty L_x^4 ([0,T))} \lesssim 1. \label{he38}
\end{align}

By \eqref{he9} and \eqref{he34}, we have
\begin{align}
\| \partial_t g \|_{L_t^\infty L_x^4([0,T))} + \left\|
\partial_{tt} g \right \|_{L_t^\infty L_x^2 ([0,T))} \lesssim 1. \label{he39}
\end{align}

Using \eqref{he38}, \eqref{he39} and \eqref{ge8}, we obtain
\begin{align}
\left\| \Delta g \right \|_{L_t^\infty L_x^2([0,T))} \lesssim 1. \label{he40}
\end{align}
Also by Hardy's inequality we get $\|\frac 1r \partial_r g \|_{L_t^\infty L_x^2([0,T))} \lesssim 1$ and hence
\begin{align}
\| \partial_{rr} g \|_{L_t^\infty L_x^2([0,T))} \lesssim 1. \label{he41a}
\end{align}

By \eqref{he7}, \eqref{he36}, \eqref{he38}, \eqref{he40}, \eqref{he41a}, it follows that
\begin{align}
\left \| \nabla A \right \|_{L_t^\infty L_x^4 ([0,T))}
+ \left\| \Delta A \right \|_{L_t^\infty L_x^2([0,T))} \lesssim 1. \label{he41}
\end{align}
Also it is not difficult to check that
\begin{align}
\left\|
\Delta \left(
\int_0^{g(r,t)}
\Bigl( 3B^{\frac 32} + B^{-\frac 12} - B^{-\frac 32} \Bigr) dy \right)
\right\|_{L_t^\infty L_x^2([0,T))} \lesssim 1. \label{he42}
\end{align}

From \eqref{he7}, \eqref{he9} and \eqref{he34}, we get
\begin{align}
\left\| \partial_{tt} A \right \|_{L_t^\infty L_x^2([0,T))}
+ \left\| \partial_{tt} A \right \|_{L_t^\infty L_x^{\frac{10}3}([0,T))} \lesssim 1.
\label{he43}
\end{align}

Differentiating \eqref{he13a} in time, we get
\begin{align}
\square_5 \bigl( \partial_{ttt} \Phi \bigr)& =
\Bigl( \frac 32 A- \frac 32 + \frac 12 A^{-1} - \frac 1 2 A^{-2} \Bigr) \partial_{ttt} \Phi \notag \\
& \qquad + (2A^{-3} -A^{-2} +3) \partial_t A \cdot \partial_{tt} \Phi \notag \\
& \qquad + (A^{-3} -\frac 12 A^{-2} +\frac 32) \partial_{tt}A \partial_t \Phi \notag \\
& \qquad + (-3 A^{-4} + A^{-3} ) (\partial_t A)^2 \partial_t \Phi. \label{he44}
\end{align}

By Strichartz, \eqref{he34}, \eqref{he43} and Sobolev, we get,
\begin{align}
 & \left\| \partial_{ttt} \Phi \right \|_{L_t^\infty \dot H_x^1([0,T))}
 + \left\| \partial_{tttt} \Phi \right \|_{L_t^\infty L_x^2([0,T))} \notag\\
 \lesssim & \left\| \partial_{ttt} \Phi(0) \right\|_{\dot H_x^1 (\mathbb R^5)}
 + \left\| \partial_{tttt} \Phi(0) \right \|_{L_x^2 (\mathbb R^5)} \notag \\
 & \qquad + \left\| (A-1) \partial_{ttt} \Phi \right\|_{L_t^1 L_x^2([0,T))}
 + \left\| \partial_t \Phi \cdot \partial_{tt} \Phi \right \|_{L_t^1 L_x^2([0,T))} \notag\\
 & \qquad + \left\| \partial_{tt} A \cdot \partial_t \Phi \right \|_{L_t^1 L_x^2([0,T))}
 + \left\| \partial_t \Phi \right\|^3_{L_t^3 L_x^6([0,T])} \notag \\
 \lesssim & 1.
 \label{he45}
 \end{align}

 By \eqref{he45}, \eqref{ge36}, we get
 \begin{align}
   \left\| \partial_{tt} \Delta \Phi \right \|_{L_t^\infty L_x^2([0,T))}
 & \lesssim \, 1+ \left\| \partial_t \Bigl( (3A^2+A^{-1} -A^{-2} ) \partial_t \Phi \Bigr)
  \right\|_{L_t^\infty L_x^2([0,T))} \notag \\
 & \lesssim \, 1. \label{he46}
  \end{align}

  Using \eqref{ge36} again with the estimates \eqref{he46} and  \eqref{he42}, we finally obtain
  \begin{align*}
  \left\| \Delta^2 \Phi \right \|_{L_t^\infty L_x^2([0,T))} \lesssim 1.
  \end{align*}

  In a similar way we have the estimate
  \begin{align*}
  \left\| \partial_t \nabla \Delta \Phi \right\|_{L_t^\infty L_x^2([0,T))} \lesssim 1.
  \end{align*}

  Hence we have established
  \begin{align}
  & \left\| \partial_{ttt} \Phi \right\|_{L_t^\infty H_x^1([0,T))}
  + \left\| \partial_{tttt} \Phi \right \|_{L_t^\infty L_x^2([0,T))} \notag \\
  & \qquad + \left \| \partial_{tt} \Phi \right\|_{L_t^\infty H_x^2([0,T))}
  + \left\| \Phi \right \|_{L_t^\infty H_x^4([0,T))} \notag \\
  & \qquad + \left\| \partial_t \Phi \right\|_{L_t^\infty H_x^3([0,T))} \lesssim 1.
  \label{he48}
  \end{align}
  This proved \eqref{he1}.

  We are now ready to prove \eqref{he2}.

  By \eqref{he36}
  \begin{align}
  \left\| \langle x \rangle g(t) \right\|_{L_t^\infty L_x^\infty([0,T))} \lesssim 1. \label{he49}
  \end{align}

  By \eqref{he9}, \eqref{he48}, Sobolev embedding and radial Sobolev embedding, we have
  \begin{align}
  \left\| \langle x \rangle \partial_t g \right\|_{L_t^\infty L_x^\infty([0,T))}
  & \lesssim \left\| \langle x \rangle \partial_t \Phi \right \|_{L_t^\infty L_x^\infty([0,T))} \notag \\
  & \lesssim \left \| \partial_t \Phi \right \|_{L_t^\infty H_x^3([0,T))} \notag \\
  & \lesssim 1. \label{he50}
  \end{align}

  In a similar way, by using \eqref{ge62}, we get
  \begin{align}
  \left\| \langle x \rangle \partial_r g \right \|_{L_t^\infty L_x^\infty([0,T))} \lesssim 1.
  \label{he51}
  \end{align}

  Now \eqref{he2} clearly follows from \eqref{he49}--\eqref{he51}.

\appendix

\section{Some technical estimates}
In this appendix we collect some useful technical estimates. Some of these estimates
are rather pedestrian. Nevertheless,  we include all
the details here for the sake of completeness.

The following radial Sobolev embedding is well-known and dates back 
to Strauss \cite{Strauss77}. We will often use it without explicit mentioning.
\begin{lem}[Radial Sobolev embedding] 
Suppose $d\ge 2$ and $h:\mathbb R^d \to \mathbb R$ is radial. If $h\in C_c^{\infty}(\mathbb R^n)$,
then for some constant $C_d>0$ depending only on the dimension $d$, we have
\begin{align*}
r^{\frac {d-1} 2} |h(r) | \le C_d \| h \|_{H^1(\mathbb R^d)}, \quad \forall\, r>0,
\end{align*}
\end{lem}
\begin{proof}
Use the identity $h(r)^2 = -2 \int_r^{\infty} h(\rho) \partial_{\rho} h(\rho) d\rho$ and
observe $r^{d-1} \le \rho^{d-1}$. 
\end{proof}

In the rest of this section, we shall show that at $t=0$, under the assumption
that $(g_0, g_1) \in H^4_{\operatorname{rad}}(\mathbb R^5)
\times H^4_{\operatorname{rad}}(\mathbb R^5)$, we have
\begin{align*}
\sum_{j=0}^3\|\partial_t^{j} \Phi(t=0) \|_{H^1(\mathbb R^5)} 
+ \| \partial_t^4 \Phi (t=0 ) \|_{L_x^2(\mathbb R^5)} <\infty.
\end{align*}
These were used in Section $4$ and Section $5$.

We now give the details. We will proceed in 8 steps.

Recall that 
\begin{align*}
\Phi(r,t) =\int_0^{g(r,t)} (1+ 
\frac{2 \sin^2 (r y + \phi(r) )} {r^2} )^{\frac 12} dy + r^{-3} \phi_{\gtrsim 1}(r),
\end{align*}
where $\phi(r)=N_1 \pi$ for $r\le 1$ and $\phi(r)=0$ for $r\ge 2$.  Recall 
$f(r,t)= \phi(r) +rg(r,t)$ and
\begin{align*}
E(t) & = \frac 1 2 \int_{0}^\infty  (1+ \frac{2\sin^2 f} {r^2} )
\Bigl( (\partial_t f)^2 + (\partial_r f)^2 \Bigr) r^2 dr
 + \int_{0}^\infty \frac{\sin^2 f}{r^2}
\Bigl( 1+ \frac{\sin^2 f}{2r^2} \Bigr) r^2 dr.
\end{align*}

Assume $(g, \partial_t g)_{t=0} =(g_0, g_1)\in H^4_{\operatorname{rad}}(\mathbb R^5)
\times H^3_{\operatorname{rad}} (\mathbb R^5)$.  By Hardy we have
$\| \frac {g_0} r \|_{L_x^2(\mathbb R^5)} <\infty$. By Sobolev embedding we have
$\| g_0\|_{\infty} +\|\nabla g_0\|_{\infty}+ \| g_1\|_{\infty}<\infty$. 

1. $\|( \partial_r f )(t=0)\|_{L^2(\mathbb R^3)} + \| (\partial_t f) (t=0)\|_{L_x^2(\mathbb R^3)}<\infty$. 

Since $\partial_r f \bigr|_{t=0}= \phi^{\prime}(r) + r \partial_r g_0 +g_0$, we have
\begin{align*}
\| \partial_r f (t=0) \|_{L^2_x(\mathbb R^3)}
& \lesssim 1 + \| \partial_r g_0 \|_{L^2_x(\mathbb R^5)} + \| \frac {g_0} r \|_{L^2_x(\mathbb R^5)} \\
& \lesssim 1 + \| g_0 \|_{H^1(\mathbb R^5)} <\infty.
\end{align*}
The estimate for $\| \partial_t f\|_{L^2(\mathbb R^3)}$ is similar and therefore omitted.

2.  $E(0)<\infty$. 

We first consider the term $\int_{0}^\infty \frac{\sin^2 f}{r^2}
( 1+ \frac{\sin^2 f}{2r^2} ) r^2 dr$.  Clearly the contribution of the part $r\sim 1$ is bounded.
Therefore we only need to consider $0<r\le 1$ and $r\ge 2$.  Then
(below for simplicity of notation $f$, $g$ and $\partial_t f$ will all be evaluated at $t=0$)
\begin{align}
\int_{0}^{\infty} \frac{\sin^2 f}{r^2}
( 1+ \frac{\sin^2 f}{2r^2} ) r^2 dr 
&\lesssim\; 1 + \int_0^{\infty} g^2 r^2 dr + \int_0^{\infty} g^4 r^2 dr \notag \\
& \lesssim\; 1+ \| \frac g r \|_{L^2_x(\mathbb R^5)}^2 + \|
\frac {g^2} {r} \|_{L_x^2(\mathbb R^5)}^2 \notag \\
& \lesssim \; 1+ \| \nabla g \|_{L_x^2(\mathbb R^5)}^2 
+ \| \nabla (g^2) \|^2_{L_x^2(\mathbb R^5)}  \label{Se140a} \\
& \lesssim \; 1+ \|g\|_{H^4(\mathbb R^5)}^4 <\infty. \notag 
\end{align}
Next for the first term in $E$, we first deal with $r\ge 1$:
\begin{align*}
\int_{1}^{\infty}  (1+ \frac{2\sin^2 f} {r^2} )
( (\partial_t f)^2 + (\partial_r f)^2 ) r^2 dr
\lesssim \| \partial_t f \|_{L^2(\mathbb R^3)}^2 + \| \partial_r f\|_{L^2(\mathbb R^3)}^2
<\infty.
\end{align*}
For the part $0<r\le 1$ we have
\begin{align*}
\int_{0}^1 & (1+ \frac{2\sin^2 f} {r^2} )
( (\partial_t f)^2 + (\partial_r f)^2 ) r^2 dr \notag \\
& \lesssim \int_0^1 (1+g^2) ( r^2 (\partial_t g)^2 + r^2 (\partial_r g)^2 +g^2) r^2 dr 
\notag \\
& \lesssim  1+\int_0^1 g^2(1+g^2) r^2 dr+ 
(1+\|g\|_{\infty}^2) (\| \partial_t g\|_{L^2_x(\mathbb R^5)}^2
+ \| \partial_r g\|_{L_x^2(\mathbb R^5)}^2)<\infty.
\end{align*}

3. $\| \nabla \Phi(t=0) \|_{L_x^2(\mathbb R^5)} <\infty$.

First observe that\footnote{Recall that $\phi_{\gtrsim 1}$ can vary from line to line.} 
$|\Phi(r,0)|\lesssim |g(r,0)|+|g(r,0)|^2+ r^{-3} |\phi_{\gtrsim 1}(r)| $, and
\begin{align*}
\| \frac {\Phi} r \|_{L_x^2(\mathbb R^5)}
\lesssim 1 + \| \frac {g} r \|_{L_x^2(\mathbb R^5)} + \| 
\frac {g^2 } r \|_{L_x^2(\mathbb R^5)} <\infty,
\end{align*}
where we have used \eqref{Se140a}.

Next by a simple change of variable $ry\to y$, we rewrite $\Phi$ as
\begin{align*}
\Phi(r,t) = r^{-2} \int_0^{rg(r,t)} (r^2+2\sin^2(y +\phi(r) ) )^{\frac 12} dy +r^{-3}
\phi_{\gtrsim 1}(r).
\end{align*}
Then
\begin{align*}
\partial_r \Phi &= - \frac 2 r \Phi +r^{-3}
\phi_{\gtrsim 1}(r) 
+r^{-2} \partial_r (rg) (r^2 +2\sin^2(rg +\phi(r) ) )^{\frac 12} \notag \\
& \quad r^{-2} 
\int_0^{rg} (r^2+2\sin^2(y+\phi) )^{-\frac 12}
(r +\sin(2y+2\phi) \phi^{\prime}(r) ) dy.
\end{align*}
If $r\ge \frac 12$, then it follows that
\begin{align*}
|\partial_r \Phi| \lesssim r^{-1} |\Phi| +r^{-3} |\phi_{\gtrsim 1} (r)|
+|\partial_r g |+ r^{-1} |g|.
\end{align*}
If $0<r<\frac 12$, then 
\begin{align*}
|\partial_r \Phi| & \lesssim r^{-1} |\Phi| +r^{-2}| \partial_r (rg)|
\cdot (r +r |g|) + r^{-1} |g| \notag \\
& \lesssim r^{-1} |\Phi| +r^{-1} |g|+ r^{-1} g^2 + |g| |\partial_r g| + |\partial_rg |.
\end{align*}
Thus we have $\| \partial_r \Phi \|_{L_x^2(\mathbb R^5)} <\infty.$

4. $\| \partial_t \Phi (t=0) \|_{H^1(\mathbb R^5)} <\infty$.

First observe that
\begin{align} \label{app_e40-1a}
\partial_t \Phi= \frac 1r \partial_t f (1+\frac {2\sin^2 f } {r^2} )^{\frac 12}.
\end{align}
Thus
\begin{align} \label{app_e40a}
\partial_t \Phi \bigr|_{t=0} = g_1\cdot (1+\frac {2\sin^2 (rg_0 +\phi(r) )} {r^2} )^{\frac 12}.
\end{align}
Since $g_1 \in H^3(\mathbb R^5)$ and $g_0 \in H^4(\mathbb R^5)$, we clearly have
$\| \frac {2\sin^2 (rg_0+\phi(r) )} {r^2}\|_{\infty} \lesssim 1$, and
\begin{align*}
\| \partial_t \Phi(t=0) \|_{L_x^2(\mathbb R^5)} <\infty.
\end{align*}
To bound the $\dot H^1$-norm, we consider first the regime $r\ge \frac 12$. Denote
\begin{align*}
B_0= 1+ \frac{ 2 \sin^2 (rg_0 + \phi(r) ) }{r^2}.
\end{align*}
Clearly for $r\ge \frac 12$, using $\|g_0\|_{\infty} +\| \nabla g_0\|_{\infty}\lesssim 1$,
we have
\begin{align}
|\partial_r B_0| \lesssim 1+ r^{-2} (| \partial_r (r g_0 )| + |\phi^{\prime}(r)|) 
\lesssim 1. \label{app_e31a}
\end{align}
Next for $r<\frac 12$, we have
\begin{align*}
B_0= 1+\frac {2 \sin^2 (rg_0 )} {(rg_0)^2} g_0^2 = 1+ \tilde G(rg_0) g_0^2,
\end{align*}
where $\tilde G( z) = 2 z^{-2} \sin^2 (z)$ has bounded derives of all orders. It follows
that
\begin{align}
|\partial_r B_0 | &\lesssim  |\partial_r (rg_0)| g_0^2 + |\partial_r g_0 | \cdot |g_0| \notag \\
& \lesssim r |\partial_r g_0| g_0^2 + |g_0|^3 + |\partial_r g_0| \cdot |g_0| 
 \lesssim  1. \label{app_e31b}
\end{align}
where we used again the fact  $\|g_0\|_{\infty}
+\|\nabla g_0\|_{\infty} \lesssim
1$.  It follows that
\begin{align*}
\| \partial_r ( \partial_t \Phi) \|_{L_x^2(\mathbb R^5)}
& \lesssim \| \partial_r g_1\|_{L_x^2(\mathbb R^5)}
+\|g_1 B_0^{-\frac 12} \partial_r B_0 \|_{L_x^2(\mathbb R^5)} \lesssim 1.
\end{align*}

5. Denote $A= 1+\frac {2\sin^2 \tilde f_0 }{r^2}$ where $\tilde f_0= \phi(r)+rg_0$. Then
\begin{align} \label{app_e32a}
\| A\|_{\infty} + \| \partial_r A \|_{\infty} \lesssim 1, \quad
\| \Delta_5 A \|_{L_x^2(\mathbb R^5)} 
+\| \Delta_5 A \|_{L_x^{10}(\mathbb R^5)} \lesssim 1.
\end{align}

 For $r\ge \frac 12$, we use \eqref{app_e31a}. For $r<\frac 12$ we use \eqref{app_e31b}.
Thus the first inequality is obvious. The second inequality follows from a similar computation.
One should note that by Sobolev embedding,
\begin{align*}
\| \Delta_5 g_0\|_{L_x^{10}(\mathbb R^5)} \lesssim \| g_0 \|_{H^4(\mathbb R^5)} <\infty.
\end{align*}

6.  $\| (\partial_{tt}\Phi) (t=0) \|_{H^1(\mathbb R^5)} <\infty$.

First observe that
\begin{align}
\partial_{tt} \Phi
= \frac 1 r \partial_{tt} f
(1+\frac {2\sin^2 f} {r^2} )^{\frac 12}
+ \frac 1 {r^3} (\partial_t f)^2 (1+ \frac {2\sin^2 f}{r^2} )^{-\frac 12}
\sin (2f). \label{app_e30a}
\end{align}

We first consider the second term on the RHS. Since $\partial_t f \bigr|_{t=0}= r g_1$, we have
\begin{align*}
\frac 1 {r^3} &(\partial_t f)^2 (1+ \frac {2\sin^2 f}{r^2} )^{-\frac 12}
\sin (2f) \bigr|_{t=0} \notag \\
=& r^{-1} g_1^2 (1+ \frac {2\sin^2( \phi+rg_0)} {r^2} )^{-\frac 12}
\sin(2\phi+2rg_0).
\end{align*}
For any $r>0$, it is not difficult to check that  (note below that for $r\le \frac 12$, 
one can write $ r^{-1}  (1+ \frac {2\sin^2( \phi+rg_0)} {r^2} )^{-\frac 12}
\sin(2\phi+2rg_0) = \tilde F(rg_0) g_0$, where $\tilde F$ has bounded derivatives)
\begin{align*}
& | \Bigl(
 r^{-1}  (1+ \frac {2\sin^2( \phi+rg_0)} {r^2} )^{-\frac 12}
\sin(2\phi+2rg_0) \Bigr) | \lesssim 1+|g_0(r)|, \\
&
|\partial_r \Bigl(
 r^{-1}  (1+ \frac {2\sin^2( \phi+rg_0)} {r^2} )^{-\frac 12}
\sin(2\phi+2rg_0) \Bigr) | \notag \\
& \qquad\qquad\qquad\lesssim 1+|g_0(r)|
+|g_0| |\partial_r(rg_0(r))|+|\partial_r g_0(r)|.
\end{align*}
It follows easily that
\begin{align*}
 \| g_1^2  r^{-1} (1+ \frac {2\sin^2( \phi+rg_0)} {r^2} )^{-\frac 12}
\sin(2\phi+2rg_0) \|_{H^1(\mathbb R^5)} \lesssim 1.
\end{align*}
It remains for us to check the first term on the RHS in \eqref{app_e30a}. 
Note that by \eqref{app_e32a} the factor $(1+\frac {2\sin^2 f} {r^2} )^{\frac 12}
$ is harmless for us when estimating the $H^1$-norm.  Therefore we only need
to focus on the estimate of $\| \frac 1 r \partial_{tt} f \|_{H^1(\mathbb R^5)}$.
Observe
that
\begin{align} \label{app_ge7-1a}
\frac 1 r \partial_{tt} f =  \partial_{tt} g= \square_5 g + \Delta_5 g.
\end{align}
Clearly $\| \Delta_5 g_0\|_{H^1(\mathbb R^5)} \lesssim 1$. For 
$\square_5 g$ we use \eqref{ge7}:
\begin{align}
\square_5 g & = \frac {\phi_{<1}} {1+\tilde F_0(rg) g^2} \Bigl( \tilde F_1(rg) g^3 + \tilde F_2(rg) g^5 \notag\\
&\qquad- \tilde F_3(rg) \cdot g \cdot \bigl( (\partial_t g)^2 - (\partial_r g)^2 \bigr) \notag \\
& \qquad + \tilde F_4(rg) \cdot g^4 \cdot r \partial_r g \Bigr) \notag \\
& \qquad + \phi_{>1} \cdot \frac 2 {r^2} g + \frac 1 r \Delta_3 \phi \notag \\
& \qquad + \frac 1 r \phi_{>1} \cdot N(r, \phi+rg, (\phi+rg)^\prime), \label{app_ge7}
\end{align}
where $\tilde F_i (x) =F_i(x^2)$ and $F_i$ has bounded derivatives of all orders. Clearly
by using radial Sobolev embedding and $\|g_0\|_{H^4} \lesssim 1$, we have
\begin{align*}
\| \partial_r ( F_i (r^2 g_0^2) ) \|_{\infty}
\lesssim \| \partial_r (r^2 g_0^2) \|_{\infty} \lesssim 1.
\end{align*}
By a tedious calculation, it is not difficult to check then that the RHS of
\eqref{app_ge7} all have bounded $H^1(\mathbb R^5)$-norm. Thus 
$\| \square_5 g_0\|_{H^1(\mathbb R^5)} \lesssim 1$ and consequently
 $\| (\partial_{tt}\Phi) (t=0) \|_{H^1(\mathbb R^5)} <\infty$.

7.  $\| (\partial_{ttt}\Phi) (t=0) \|_{H^1(\mathbb R^5)} <\infty$.

Here we use \eqref{ge36}:
\begin{align*}
\partial_{ttt} \Phi \bigr|_{t=0} &= \Delta_5 \partial_t \Phi\bigr|_{t=0}- \frac 3 2 \partial_t \Phi
\bigr|_{t=0}
+ \frac 12 \partial_t g  (3 B^{\frac 32}\bigr|_{t=0} + B^{-\frac 12}\bigr|_{t=0} - B^{-\frac 32}
\bigr|_{t=0} ), \notag \\
&=\Delta_5 \partial_t \Phi\bigr|_{t=0}- \frac 3 2 \partial_t \Phi
\bigr|_{t=0}
+ \frac 12 g_1 (3 A^{\frac 32} +A^{-\frac 12} -A^{-\frac 32}),
\end{align*}
where $A$ is the same as in \eqref{app_e32a}.  By \eqref{app_e32a}, the last
term above clearly is bounded in $H^1(\mathbb R^5)$.  Also in Step 4 we have shown
$\| \partial_t \Phi \bigr|_{t=0} \|_{H^1(\mathbb R^5)} <\infty$.  By 
\eqref{app_e40a}, we have
\begin{align*}
\Delta_5 ( \partial_t \Phi\bigr|_{t=0})
=\Delta_5(g_1 A^{\frac 12}) 
=\Delta_5 (g_1 ) A^{\frac 12}+ 2 \partial_r g_1 \cdot \partial_r (A^{\frac 12})
+g_1 \Delta_5 (A^{\frac 12} ).
\end{align*}
Clearly by \eqref{app_e32a}, it follows that $\| \Delta_5 ( \partial_t \Phi\bigr|_{t=0})
\|_{L_x^2(\mathbb R^5)} \lesssim 1$.

8.  $\| (\partial_{tttt}\Phi) (t=0) \|_{L_x^2(\mathbb R^5)} <\infty$.
Here we use again \eqref{ge36}:
\begin{align*}
\partial_{tttt} \Phi \bigr|_{t=0} 
&=\Delta_5 \partial_{tt} \Phi\bigr|_{t=0}- \frac 3 2 \partial_{tt} \Phi
\bigr|_{t=0}
+ \frac 12 \partial_{tt} g \bigr|_{t=0} (3 A^{\frac 32} +A^{-\frac 12} -A^{-\frac 32}) \notag \\
&\qquad + \frac 12 g_1 (\frac 9 2 A^{\frac 12}
-\frac 12 A^{-\frac 32} +\frac 32 A^{-\frac 52} )
\cdot \frac { \sin (2rg_0 +2\phi) }{r^2} \cdot (2r g_1). 
\end{align*}
By the calculation in Step 6, we have $\|\partial_{tt} g \bigr|_{t=0} \|_{L_x^2(\mathbb R^5)}
\lesssim 1$. The last three terms above clearly is $L_x^2(\mathbb R^5)$-bounded.

We now only to estimate $ \| \Delta_5 \partial_{tt} \Phi \bigr|_{t=0} \|_{L_x^2(\mathbb R^5)}$.
By \eqref{app_e40-1a}, we have
\begin{align*}
\partial_{tt} \Phi \bigr|_{t=0}
= \frac 1 r \partial_{tt} f \bigr|_{t=0}A^{\frac 12}
+r^{-1} g_1^2 A^{-\frac 12} \sin 2 \tilde f_0,
\end{align*}
where $\tilde f_0=\phi +rg_0$.  Clearly by \eqref{app_e32a} and
$\|g_0\|_{H^4} +\|g_1\|_{H^3}\lesssim 1$, we have
\begin{align*}
\| \Delta_5 (   g_1^2 A^{-\frac 12} \frac {\sin 2 \tilde f_0} r ) \|_{L_x^2(\mathbb R^5)}
\lesssim 1. 
\end{align*}
By \eqref{app_ge7-1a}, we have
\begin{align*}
\frac 1 r \partial_{tt} f \bigr|_{t=0} A^{\frac 12}
=  A^{\frac 12} (\square_5 g + \Delta_5 g)\bigr|_{t=0}.
\end{align*}
By \eqref{app_e32a}, we have
\begin{align*}
\| \Delta_5 ( A^{\frac 12} \Delta_5 g_0) \|_{L_x^2(\mathbb R^5)}
\lesssim \| g_0\|_{H^4} + \| \Delta_5 (A^{\frac 12} )\|_{L_x^4(\mathbb R^5)}
\| \Delta_5 g_0\|_{L_x^4(\mathbb R^5)} <\infty.
\end{align*}
Similarly we have (below we used the simple inequality
$\|h\|_{L_x^4(\mathbb R^5)} \lesssim \| h\|_{L_x^2(\mathbb R^5)}
+ \| \Delta_5 h \|_{L_x^2(\mathbb R^5)}$)
\begin{align*}
\| \Delta_5 ( A^{\frac 12} \square_5 g_0 ) \|_{L_x^2(\mathbb R^5)}
&\lesssim \| \square_5 g_0 \|_{L_x^4(\mathbb R^5)}
+ \| \nabla \square_5 g_0\|_{L_x^2(\mathbb R^5)}
+ \| \Delta_5 \square_5 g_0\|_{L_x^2(\mathbb R^5)} \notag \\
& \lesssim  \| \square_5 g_0\|_{L_x^2(\mathbb R^5)}
+ \|\Delta_5 \square_5 g_0\|_{L_x^2(\mathbb R^5)}.
\end{align*}
In Step 6 (see the estimates near \eqref{app_ge7-1a}, we have estimated
$\| \square_5 g_0 \|_{H^1(\mathbb R^5)}$. Thus we only need to deal with
the term $\| \Delta_5 \square_5 g_0\|_{L_x^2(\mathbb R^5)}$. 
By using \eqref{app_ge7} and a tedious computation, it is not difficult
to check that the RHS of \eqref{app_ge7} has finite $H^2(\mathbb R^5)$-norm.
This then completes the estimate of $\| \partial_{tttt} \Phi\bigr|_{t=0}\|_{L_x^2(\mathbb R^5)}$.

\section{The continuity argument} \label{S:cont}
In this appendix we give more details of the continuity argument in  the 
derivation of \eqref{he12}. 
Recall the main equation:
\begin{align}
\square_5 ( \partial_t \Phi)
=\Bigl( \frac 32 + \frac 12 A^{-2} \Bigr) (A-1) \partial_t \Phi,
\end{align}
and 
\begin{align} \label{app_e5.8}
|A-1| \lesssim \min\{ r^{-\frac 32}, r^{-4} \}.
\end{align}

Now denote $u=\partial_t \Phi$.  Our goal is to show that on the interval $[0,T_1]$ ($T_1<T$
can be arbitrarily close to $T$), we have
\begin{align} \label{app_eNew1}
\| u \|_{L_t^3 L_x^3 ([0,T_1]) }+ \|u \|_{C_t^0 \dot H_x^{\frac 12} ([0,T_1])} \lesssim 1,
\end{align}
where the implied constant is independent of $T_1$.

To this end we decompose $[0,T_1] = \bigcup_{i=0}^{N_0} [t_i, t_{i+1}]$, where
$t_0=0$, $t_{N_0+1} =T_1$, and $t_{i+1}-t_i$ will be taken sufficiently small.
The needed smallness will become clear in the argument below.

First observe that by using the estimates in Section 4, we have
\begin{align} \label{app_e5.91}
\| P_{<1} u 
\|_{L_t^3 L_x^3([0,T_1])} \lesssim \| u \|_{L_t^{\infty}L_x^2 ([0,T_1])} 
=\| \partial_t \Phi \|_{L_t^{\infty} L_x^2 ([0,T_1])}
\lesssim 1. 
\end{align}

By Strichartz (Lemma \ref{lem_S1}) and \eqref{app_e5.8}, we have on
each $[t_i, t_{i+1}]$, 
\begin{align}
&\left \| P_{\ge 1} u \right\|_{ L_t^3 L_x^3 ( [t_i, t_{i+1} ])}
+ \left\| P_{\ge 1} u \right\|_{ C_t^0\dot H^{\frac 12}_x([t_i,t_{i+1}])}
+\left \| P_{\ge 1} \partial_t u \right \|_{C_t^0 \dot  H_x^{-\frac 12} ([t_i,t_{i+1}])}
\\
 &
\lesssim \left \| P_{\ge 1} u(t_i ) \right\|_{\dot H_x^{\frac 12} }
 + \left\| P_{\ge 1} \partial_{t} u(t_i)
\right \|_{\dot H_x^{-\frac 12} } 
 + \left\| (A-1) u \right \|_{L_t^{\frac 32} L_x^{\frac 32} ([t_i, t_{i+1}])} \notag \\
& \lesssim \left \| P_{\ge 1} u(t_i ) \right\|_{\dot H_x^{\frac 12} }
 + \left\| P_{\ge 1} \partial_{t} u(t_i)
\right \|_{\dot H_x^{-\frac 12} } 
  + \| (A-1) \|_{L_t^3L_x^3([t_i,t_{i+1}])} \cdot
\| u \|_{L_t^3 L_x^3([t_i, t_{i+1}])} \notag \\
& \lesssim \left \| P_{\ge 1} u(t_i ) \right\|_{\dot H_x^{\frac 12} }
 + \left\| P_{\ge 1} \partial_{t} u(t_i)
\right \|_{\dot H_x^{-\frac 12} } 
+ (t_{i+1}-t_i)^{\frac 13} \| u \|_{L_t^3 L_x^3 ([t_i, t_{i+1}])}.
\end{align}
Clearly if $(t_{i+1}-t_i)$ is sufficiently small, we have (using \eqref{app_e5.91})
\begin{align*}
&\left \| P_{\ge 1} u \right\|_{ L_t^3 L_x^3 ( [t_i, t_{i+1} ])}
+ \left\| P_{\ge 1} u \right\|_{ C_t^0\dot H^{\frac 12}_x([t_i,t_{i+1}])}
+\left \| P_{\ge 1} \partial_t u \right \|_{C_t^0 \dot  H_x^{-\frac 12} ([t_i,t_{i+1}])}
\notag \\
\lesssim&\;  \left \| P_{\ge 1} u(t_i ) \right\|_{\dot H_x^{\frac 12} }
 + \left\| P_{\ge 1} \partial_{t} u(t_i)
\right \|_{\dot H_x^{-\frac 12} } +1.
\end{align*}

Clearly by using the above estimate and iterating from $i=0$ to $i=N_0$ (for the base step
$i=0$, one can use the estimates in Appendix A to obtain
$ \left \| P_{\ge 1} u(t=0) \right\|_{\dot H_x^{\frac 12} }
 + \left\| P_{\ge 1} \partial_{t} u(t=0)
\right \|_{\dot H_x^{-\frac 12} }  \lesssim 1$), we
can obtain the estimate \eqref{app_eNew1}.

\section{Additional estimates} \label{S:add1}
This appendix is for the estimates in \eqref{he33} and \eqref{he38}.

For $r\le \frac 12$,  we have $f= N_1\pi+ rg$. By using \eqref{ge62}, we have
\begin{align} \label{C1e1}
&\partial_r \Phi(r,t) + \frac 2 r \Phi (r,t) \notag \\
 =&  \left( 1+ \frac{2\sin^2 f} {r^2} \right)^{\frac 12}
\cdot (\partial_r g + \frac g r)
 + \frac 1 r \int_0^{g(r,t)} \left( 1+ \frac{2\sin^2(ry)} {r^2} \right)^{-\frac 12} dy. 
\end{align}
By \eqref{he21} and \eqref{he22}, we have for $r\le \frac 12$,
\begin{align} \label{C1e3}
|\Phi(r) | \lesssim r^{-\frac 12}, \quad |g(r) |\lesssim r^{-\frac 14}.
\end{align}
Thus  for $r\le \frac 12$,  plugging \eqref{C1e3} into \eqref{C1e1}, we obtain
\begin{align*}
|\partial_r g| \lesssim |\partial_r \Phi| + r^{-\frac 32}.
\end{align*}
This estimate is used in \eqref{he33}.

Next we turn to \eqref{he38}.  By \eqref{he35} and \eqref{he36}, we have
\begin{align*}
|\Phi(r)|\lesssim \langle r \rangle^{-2}, \quad |g(r)|\lesssim \langle r\rangle^{-2}.
\end{align*}
By \eqref{C1e1}, we then have
\begin{align*}
| \partial_r g| \lesssim |\partial_r \Phi|+ r^{-1} \langle r\rangle^{-2}.
\end{align*}
Thus $\| \partial_r g \|_{L_t^{\infty} L_x^4 } \lesssim 1$.


\begin{thebibliography}{m}
\bibitem{AN84}
G. Adkins and C. Nappi. Phys. Lett. B 137 (1984) 251--256.


\bibitem{Bi07}
P. Bizo\'n, T. Chmaj, A. Rostworowski, Asymptotic stability
of the Skyrmion, Phys. Rev. D 75, 121702 (2007)


\bibitem{Bi00}
P. Bizo\'n, Equivariant self-similar wave maps from
Minkowski spacetime into 3-sphere,   Comm. Math. Phys. 215, 45--56 (2001)

\bibitem{Bi01a}
P. Bizo\'n, T. Chmaj and Z. Tabor.
Dispersion and collapse of wave maps.
Nonlinearity 13 (2000), no. 4, 1411--1423.

\bibitem{BPST03}
N. Burq, F. Planchon, J. G. Stalker and A. S. Tahvildar-Zadeh.
Strichartz estimates for the wave and Schr\"odinger equations with the inverse-square potential.
J. Funct. Anal. , vol. 203, no. 2, pp. 519--549, 2003.

\bibitem{CB94}
W.Y. Crutchfield and J.B. Bell. Instabilities of the Skyrme model.
J. Comput. Phys. 130:234--241 (1994)

\bibitem{CSTZ98}
T. Cazenave; J. Shatah, Jalal and A.S. Tahvildar-Zadeh.
Harmonic maps of the hyperbolic space and development of singularities in wave maps
and Yang-Mills fields.
Ann. Inst. H. Poincar\'e Phys. Th\'eor. 68 (1998), no. 3, 315--349.



\bibitem{CTZ93}
D. Christodoulou and A.S. Tahvildar-Zadeh.
On the regularity of spherically symmetric wave maps.
Comm. Pure Appl. Math. 46 (1993), no. 7, 1041--1091.

\bibitem{GR10a}
D. Geba and S.G. Rajeev.
A continuity argument for a semilinear Skyrme model.
Electron. J. Differential Equations 2010, No. 86, 9 pp.

\bibitem{GR10b}
D. Geba and S.G. Rajeev.
Nonconcentration of energy for a semilinear Skyrme model.
Ann. Physics 325 (2010), no. 12, 2697--2706.

\bibitem{Gnew1}
Geba, Dan-Andrei; Grillakis, Manoussos G.
Large data global regularity for the classical equivariant Skyrme model. 
Discrete Contin. Dyn. Syst. 38 (2018), no. 11, 5537--5576.

\bibitem{Gnew2}
Geba, Dan-Andrei; Grillakis, Manoussos G.
Large data global regularity for the 2+1-dimensional equivariant Faddeev model. 
Differential Integral Equations 32 (2019), no. 3-4, 169--210.


\bibitem{Gnew4}
Geba, Dan-Andrei; Grillakis, Manoussos G.
An introduction to the theory of wave maps and related geometric problems. World Scientific Publishing Co. Pte. Ltd., Hackensack, NJ, 2017.

\bibitem{Gnew3}
Creek, Matthew
Large-Data Global Well-Posedness for the (1 + 2)-Dimensional Equivariant Faddeev Model.
Thesis (Ph.D.)--University of Rochester. 2014. 73 pp. ISBN: 978-1303-92164-3 

\bibitem{GNR11}
 Geba, Dan-Andrei; Nakanishi, Kenji; Rajeev, Sarada G. Global well-posedness and scattering for Skyrme wave maps. Commun. Pure Appl. Anal. 11 (2012), no. 5, 1923--1933.


\bibitem{Gibbons03}
G.W. Gibbons. Causality and the Skyrme model. Physics Letters B, 566:171--174, 2003.

\bibitem{KM08}
C.E. Kenig and F. Merle.
Global well-posedness, scattering and blow-up for the energy-critical focusing non-linear wave equation.
Acta Math. 201 (2008), no. 2, 147--212.

\bibitem{KM97}
S. Klainerman and M. Machedon.
On the regularity properties of a model problem related to wave maps.
Duke Math. J. 87 (1997), no. 3, 553--589.




\bibitem{KS97}
S. Klainerman and S. Selberg.
Remark on the optimal regularity for equations of wave maps type.
Comm. Partial Differential Equations 22 (1997), no. 5--6, 901--918.


\bibitem{KST_Inventione}
J. Krieger, W. Schlag and D. Tataru.
Renormalization and blow up for charge one equivariant critical wave maps.
Invent. Math. 171 (2008), no. 3, 543--615.

\bibitem{KS09}
Krieger, Joachim; Schlag, Wilhelm.
Concentration compactness for critical wave maps.
EMS Monographs in Mathematics. European Mathematical Society (EMS), Zurich, 2012.



\bibitem{RS_Annals}
I. Rodnianski and J. Sterbenz.
On the formation of singularities in the critical ${\text O}(3)$ $\sigma$-model.
Ann. of Math. (2) 172 (2010), no. 1, 187--242.

\bibitem{Shatah88}
J. Shatah.
Weak solutions and development of singularities of the ${\text SU}(2)$ $\sigma$-model.
Comm. Pure Appl. Math. 41 (1988), no. 4, 459--469.

\bibitem{Sk61a} T.H.R. Skryme. A nonlinear field theory. Proc. Roy. Soc. A 260, 127--138 (1961).

\bibitem{Sk61b} T.H.R. Skyrme. Particle states of a quantized meson field.
  Proc. Roy. Soc. A 262, 237--245 (1961).

\bibitem{Sk62} T.H.R. Skyrme. A unified field theory of mesons and baryons. Nucl. Phys. 31, 556--569
 (1962).

\bibitem{STZ94} J. Shatah and  A.S. Tahvildar-Zadeh.
On the Cauchy problem for equivariant wave maps.
Comm. Pure Appl. Math. 47 (1994), no. 5, 719--754.

\bibitem{Strauss77}
Strauss, Walter A.
Existence of solitary waves in higher dimensions.
Comm. Math. Phys. 55 (1977), no. 2, 149--162. 



\bibitem{Struwe03}
M. Struwe.
Equivariant wave maps in two space dimensions.
Dedicated to the memory of J\"orgen K. Moser.
Comm. Pure Appl. Math. 56 (2003), no. 7, 815--823.

\bibitem{ST10A}
J. Sterbenz and D. Tataru.
Energy dispersed large data wave maps in $2+1$ dimensions.
Comm. Math. Phys. 298 (2010), no. 1, 139--230.

\bibitem{ST10B}
J. Sterbenz and D. Tataru.
Regularity of wave-maps in dimension $2+1$.
Comm. Math. Phys. 298 (2010), no. 1, 231--264.

\bibitem{Tataru01}
D. Tataru.
On global existence and scattering for the wave maps equation.
Amer. J. Math. 123 (2001), no. 1, 37--77.


\bibitem{Wong11} Wong, Willie Wai-Yeung.
Regular hyperbolicity, dominant energy condition and causality for Lagrangian theories of maps. 
Classical Quantum Gravity 28 (2011), no. 21, 215008, 23 pp. 
\end{thebibliography}
\end{document}